\documentclass[12pt,a4paper,twoside,reqno]{amsart}

\usepackage{tikz}
\usepackage{slashed}
\usepackage{amsmath,amscd}
\usepackage{amssymb,amsfonts,mathrsfs}
\usepackage[driver=pdftex,margin=3cm,heightrounded=true,centering]{geometry}
\usepackage{JK}
\usepackage[colorlinks=true,linkcolor=blue,citecolor=blue]{hyperref}

\newtheorem{thm}{Theorem}[section]

\usepackage{graphicx}

\begin{document}
\title[Differentiable absorption, connections, and lifts]{Differentiable absorption of Hilbert $C^*$-modules, connections, and lifts of unbounded operators}

\author{Jens Kaad}
\address{International School of Advanced Studies (SISSA),
Via Bonomea 265,
34136 Trieste,
Italy}
\email{jenskaad@hotmail.com}

%
%
%

\subjclass[2010]{46L08, 46L87; 53C05, 47A05, 46L57, 58B34}
\keywords{Hilbert $C^*$-modules, derivations, differentiable absorption, Gra\ss mann connections, regular unbounded operators.}

\begin{abstract}
The Kasparov absorption (or stabilization) theorem states that any countably generated Hilbert $C^*$-module is isomorphic to a direct summand in the standard module of square summable sequences in the base $C^*$-algebra. In this paper, this result will be generalized by incorporating a densely defined derivation on the base $C^*$-algebra. This leads to a differentiable version of the Kasparov absorption theorem. The extra compatibility assumptions needed are minimal: It will only be required that there exists a sequence of generators with mutual inner products in the domain of the derivation. The differentiable absorption theorem is then applied to construct densely defined connections (or correspondences) on Hilbert $C^*$-modules. These connections can in turn be used to define selfadjoint and regular "lifts" of unbounded operators which act on an auxiliary Hilbert $C^*$-module.
\end{abstract}

\maketitle
\tableofcontents
\section{Introduction}
The famous Kasparov absorption theorem states that any countably generated Hilbert $C^*$-module $X$ over any $C^*$-algebra $A$ is a direct summand in a free Hilbert $C^*$-module, \cite{Kas:HSV, MiPh:ETH, Lan:HCM}. One may thus think of Hilbert $C^*$-modules as a natural generalization of finitely generated projective modules over $C^*$-algebras.

The main purpose of this paper is to prove a version of the Kasparov absorption theorem which takes into account any differentiable structure which may exist on the base $C^*$-algebra $A$. Following the scheme of noncommutative geometry, this extra differentiable structure will be encoded in a densely defined derivation $\de$ which is compatible with the adjoint operation, \cite{Con:NCG}.

One of the main applications of the Kasparov absorption theorem is to the construction of the interior Kasparov product in $KK$-theory, \cite{Kas:OFE, Bla:KOA, JeTh:EKT}. Consequently, we expect that the differentiable absorption theorem will play an important role for the current investigations of the unbounded version of the interior Kasparov product, \cite{KaLe:SFU, Mes:UCN}.  

Among the challenges which arise during the construction of the unbounded Kasparov product one encounters the following: Consider an unbounded (selfadjoint and regular) operator $D$ acting on an auxiliary Hilbert $C^*$-module $Y$ which carries an action of $A$. Suppose that $D$ implements the densely defined derivation on $A$ by taking commutators. Is it then possible to construct:
\begin{enumerate}
\item A Hermitian connection $\Na$ which is densely defined on $X$?
\item An unbounded operator $1 \ot_{\Na} D$ which is densely defined on the interior tensor product of $X$ and $Y$ and which has the \emph{formal} expression $c(\Na) + 1 \ot D$, where $c$ denotes the ``Clifford action''?
\end{enumerate}

The second purpose of this paper is to provide a detailed discussion of these problems.
\bigskip

Before we go any further, let us state the Kasparov absorption theorem. Let $H_A$ denote the standard module consisting of square summable sequences in $A$.

\begin{thm}[Continuous absorption]
There exists a bounded adjointable isometry $W : X \to H_A$.
\end{thm}

Let $P := W W^* : H_A \to H_A$ denote the associated orthogonal projection and let us choose a dense $*$-subalgebra $\sA \su A$ which is included in the domain of the derivation $\de$. Suppose now that $P$ is represented by an infinite matrix $\{P_{ij}\}$ of elements in $\sA$. We are then interested in analyzing (the operator norm of) the derivative $\de(P) := \{ \de(P_{ij}) \}$. Our first remark is that it is known from examples that $\de(P)$ need \emph{not} be a bounded operator, see \cite[Proposition 6.18]{BrMeSu:GSU} for the concrete case of the ($\te$-deformed) Hopf fibration and \cite{Kaa:SSB} for a general discussion in the commutative case. 
\medskip

The main idea of the differentiable absorption theorem is to introduce an extra bounded operator which regularizes the growth of the derivative $\de(P)$. We will accomplish this task under the following minimal assumption:

\begin{assu}
There exists a sequence $\{\xi_n\}$ of generators for $X$ such that the inner product $\inn{\xi_n,\xi_m}$ lies in $\sA$ for all $n,m \in \nn$.
\end{assu}

In order to state our result, let us introduce the notation $\sK(H_A)$ for the compact operators on the standard module $H_A$ and $\sK(H_A)_\de$ for the \emph{differentiable} compact operators. The latter Banach $*$-algebra agrees with the completion of the finite matrices over $\sA$ with respect to the norm $\| \cd \|_\de := \| \cd \| + \| \de(\cd) \|$.

\begin{thm}[Differentiable absorption]
There exists a bounded adjointable isometry $W : X \to H_A$ and a positive selfadjoint bounded operator $K : H_A \to H_A$ such that
\begin{enumerate}
\item $K P = P K$
\item $W^* K W : X \to X$ has dense image.
\item $P K \in \sK(H_A)$
\item $P K^2 \in \sK(H_A)_\de$
\end{enumerate}
where $P := W W^* : H_A \to H_A$ is the associated orthogonal projection.
\end{thm}

Our first main application of the differentiable absorption theorem is to construct a densely defined Gra\ss mann connection. To explain this result, let $\Om_\de(A) \su \sL(Y)$ denote the smallest $C^*$-subalgebra which contains $A$ and the image of the derivation $\de : \sA \to \sL(Y)$. We think of $\Om_\de(A)$ as an analogue of the continuous forms on a manifold. The Gra\ss mann connection is then \emph{formally} given by the formula $\Na_\de := P \de P$. The main problem is however to show that this expression makes sense and yields a \emph{densely defined} $\cc$-linear map on the direct summand $P H_A$ with values in the interior tensor product $P H_A \hot_A \Om_\de(A)$. This relies heavily on the differentiable absorption theorem. In order to state the properties of our Gra\ss mann connection we introduce the following pairing:
\[
\big( \cd , \cd \big) : X \ti X \hot_A \Om_\de(A) \to \Om_\de(A) \q \big(\xi, \eta \ot \om \big) := \inn{\xi,\eta} \cd \om
\]

\begin{thm}\label{t:graconint}
There exists a dense $\sA$-submodule $\sX \su X$ and a $\cc$-linear map $\Na_\de : \sX \to X \hot_A \Om_\de(A)$ which satisfies the Leibniz rule and is Hermitian, in the sense that
\begin{enumerate}
\item $\Na_\de(\xi \cd a) = \Na_\de(\xi) \cd a + \xi \ot \de(a)$
\item $\de(\inn{\xi,\eta}) = \big( \xi, \Na_\de(\eta) \big) - \big( \eta, \Na_\de(\xi) \big)^*$
\end{enumerate}
for all $\xi,\eta \in \sX$ and all $a \in \sA$.
\end{thm}

We would like to emphasize that our notion of connection is different from previous notions of connections in noncommutative geometry, see \cite[Section 8]{CuQu:AEN}, \cite[Part II, Definition 18]{Con:NDG} and \cite[Definition 1.7]{Kar:HCK}. One of the main differences is here that the range of the connection, thus the Hilbert $C^*$-module $X \hot_A \Om_\de(A)$ is not defined algebraically (we have passed to a completion of the algebraic tensor product $X \ot_A \Om_\de(A)$). This is an important difference which allows us to deal with Hilbert $C^*$-modules which are not necessarily finitely generated projective. Notice also that the context of Hilbert $C^*$-modules also allows us to formulate the second condition of Hermitianness for our connections.
\medskip

With the Gra\ss mann connection $\Na_\de$ in hand we can make sense of the following operator at the \emph{algebraic level}:
\[
1 \ot_\Na D : \sX \ot_{\sA} \sD(D) \to X \hot_A Y \q 1 \ot_{\Na} D : \xi \ot \eta \mapsto \Na_\de(\xi)(\eta) + \xi \ot D(\eta)
\]
thus $\ot_{\sA}$ denotes the tensor product of modules over $\sA$, whereas $\hot_A$ denotes the interior tensor product of Hilbert $C^*$-modules. Let now $Y^\infty$ denote the Hilbert $C^*$-module of square-summable sequences in $Y$. In order to have a well-defined (and more manageable) unbounded operator we replace $1 \ot_\Na D$ with the contraction
\[
Q \cd \T{diag}(D) \cd Q : \sD\big( \T{diag}(D) Q \big) \to Q Y^\infty
\]
where $Q := P \ot 1 : Y^\infty \to Y^\infty$ is an orthogonal projection induced by $P : H_A \to H_A$ and $\T{diag}(D) : \T{diag}(D) \to Y^\infty$ is the diagonal operator induced by $D : \sD(D) \to Y$  We are interested in understanding the properties of the contraction $Q \cd \T{diag}(D) \cd Q$. More precisely, we investigate two fundamental questions: 
\begin{enumerate}
\item Is the closure of the contraction $Q \cd \T{diag}(D) \cd Q$ selfadjoint? 
\item Is the closure of the contraction $Q \cd \T{diag}(D) \cd Q$ regular? 
\end{enumerate}
In general, the contraction need not be essentially selfadjoint: Indeed, by analyzing our construction for the half-line, we see that $Q \cd \T{diag}(D) \cd Q$ provides a symmetric extension of the Dirac operator $i \frac{d}{dt} : C_c^\infty\big( (0,\infty) \big) \to L^2\big( (0,\infty) \big)$. This Dirac operator has no selfadjoint extensions due to a mismatch of the deficiency indices. We do not have a counterexample to regularity but we strongly believe that such an example exists.

In order to solve this lack of selfadjointness (and possibly also of regularity) we modify the contraction $Q \cd \T{diag}(D) \cd Q$ by multiplying it from the left and from the right with the positive selfadjoint bounded operator with dense image, $\De := Q (K^2 \ot 1) Q : Q Y^\infty \to Q Y^\infty$. We then obtain our third main result:

\begin{thm}\label{t:selregint}
Suppose that $W : X \to H_A$ and $K : H_A \to H_A$ satisfy the properties stated in the differentiable absorption theorem. Then the closure of the unbounded operator
\[
\De Q \cd \T{diag}(D) \cd Q \De : \sD\big( \T{diag}(D) Q \De \big) \to Q Y^\infty
\]
is selfadjoint and regular.
\end{thm}

The plan of the present paper is as follows:
\medskip

In Section \ref{s:conabs} we provide a novel proof of the Kasparov absorption theorem. The usual proof consists of first stabilizing $X$ with the standard module $H_A$ and then construct a bounded adjointable operator $T : H_A \to X \op H_A$ such that both $T$ and $T^*$ have dense image. This yields a unitary isomorphism $H_A \cong X \op H_A$ by taking polar decompositions, see for example \cite[Theorem 2.3]{RaTh:KSF} or \cite[Theorem 1.4]{MiPh:ETH}. Another (and slightly more concrete) possibility is to apply a version of the Gram-Schmidt orthonormalization procedure to the generators of the Hilbert $C^*$-module (after stabilizing with the standard module), see for example \cite[Theorem 2]{Kas:HSV}. With both of these methods, it seems impossible to obtain any control on the growth of the derivative of the associated orthogonal projection $P$. Our new proof is straightforward and basically consists of choosing better and better approximations to the inverse of the infinite matrix
\[
G = \big\{ \inn{\xi_i,\xi_j} \big\} : H_A \to H_A
\]
induced by the sequence of generators. With this procedure, we do not need to stabilize $X$ by adding the standard module $H_A$.

In Section \ref{s:difabs} we give a proof of the differentiable absorption theorem. As noted above, this is only possible because our construction of the bounded adjointable isometry $W : X \to H_A$ is more explicit than the usual construction. The extra bounded operator $K : H_A \to H_A$ also has a simple description in terms of the generators of the Hilbert $C^*$-module (it is basically nothing but the operator $G$). 

In Section \ref{s:gracon} we apply the differentiable absorption theorem to construct a densely defined Gra\ss mann connection on the Hilbert $C^*$-module $X$, see Theorem \ref{t:graconint}.

In Section \ref{s:symsel} we investigate the properties of the associated symmetric lift $1 \ot_\Na D$ and we show that it need not be selfadjoint in general.

In Section \ref{s:comregunb} we analyze the following general question: Given a selfadjoint and regular operator $D : \sD(D) \to X$ and a bounded selfadjoint operator $x : X \to X$, what can we then say about the selfadjointness and regularity of the product $x D x$? This part relies on our earlier investigations with M. Lesch which led to a local-global principle for regular unbounded operators, see \cite{KaLe:LGR}.

In Section \ref{s:selreg} we provide a proof of Theorem \ref{t:selregint} which relies on the achievements of the preceding sections.
\medskip

Even though some of the concepts introduced in Section \ref{s:difabs} and Section \ref{s:gracon} are more naturally understood in the framework of operator modules we have decided to avoid this terminology in order to keep the exposition as simple as possible. We refer the reader who is interested in operator modules and their relation to Hilbert $C^*$-modules to the following two beautiful papers by D. Blecher, \cite{Ble:AHM, Ble:GHM}.

\subsection{Acknowledgements}
I wanted to thank Ludwik Dabrowski for our discussions on the example concerning the half-line and for his general encouragement.

\section{Continuous absorption}\label{s:conabs}
\emph{Throughout this section $X$ will be a countably generated Hilbert $C^*$-module over an arbitrary $C^*$-algebra $A$.}

Recall that the assumption ``$X$ is countably generated'' means that there exists a sequence $\{\xi_n\}_{n = 1}^\infty$ of elements in $X$ such that the $A$-span
\[
\T{span}_{A}\{ \xi_n \, | \, n \in \nn\} := \{ \sum_{n = 1}^N \xi_n \cd a_n \, | \, N \in \nn \, , \, a_n \in A\}
\]
is dense in $X$.

Let us fix such a sequence $\{ \xi_n \}$. Without loss of generality we may assume that the norm-estimate 
\begin{equation}\label{eq:norestgen}
\| \xi_n \| \leq \frac{1}{n}
\end{equation}
holds for all $n \in \nn$.

Let us denote the standard module over $A$ by $H_A$. Recall that $H_A$ consists of the sequences $\{a_n\}_{n = 1}^\infty$ in $A$ such that the sequence $\big\{ \sum_{n = 1}^N a_n^* a_n \}_{N = 1}^\infty$ converges in the norm on $A$. The inner product on $H_A$ is given by $\inn{ \{a_n\}, \{b_n\}} := \sum_{n = 1}^\infty a_n^* \cd b_n$ and the right action is given by $\{ a_n\} \cd a := \{ a_n \cd a\}$.

For each $N \in \nn$ define the compact operator $\Phi_N : X \to H_A$, $\Phi_N : \eta \mapsto \{ \inn{\xi_n,\eta} \}_{n = 1}^N$. The adjoint is given by $\Phi_N^* : H_A \to X$, $\Phi_N^* : \{a_n\}_{n = 1}^\infty \mapsto \sum_{n = 1}^N \xi_n \cd a_n$.

\begin{lemma}\label{l:geninc}
The sequence $\{ \Phi_N\}_{N = 1}^\infty$ converges in operator norm to a compact operator $\Phi : X \to H_A$. The adjoint $\Phi^* : H_A \to X$ coincides with the norm limit of the sequence $\{ \Phi_N^* \}_{N = 1}^\infty$.
\end{lemma}
\begin{proof}
It is enough to show that the sequence $\{ \Phi_N \}_{N = 1}^\infty$ is a Cauchy sequence in operator norm. Thus, let $N,M \in \nn$ with $M \geq N$ be given. For each $\eta \in X$ we have that
\[
\begin{split}
\big\| \Phi_M(\eta) - \Phi_N(\eta) \big\|^2 
& = \big\| \{ \inn{ \xi_n,\eta}\}_{n = N + 1}^M \big\|^2 \\
& = \big\| \sum_{n = N + 1}^M \inn{\eta,\xi_n} \cd \inn{\xi_n,\eta} \big\|
\leq \| \eta\|^2 \cd \sum_{n = N + 1}^M \frac{1}{n^2}
\end{split}
\]
where we have applied the norm estimate in \eqref{eq:norestgen}. This computation shows that
\[
\big\| \Phi_M - \Phi_N \big\| \leq \sqrt{ \sum_{n = N + 1}^M \frac{1}{n^2}}
\]
The sequence $\{ \Phi_N \}_{N = 1}^\infty$ is therefore a Cauchy sequence in operator norm. 
\end{proof}

Define the positive compact operator
\[
G := \Phi \Phi^* : H_A \to H_A
\]
%
%
%
For each $n \in \nn$ define the positive selfadjoint operator 
\[
G_n := (G + 1/n)^{-1} : H_A \to H_A
\]
To ease the notation later on, let also $G_0 := 0$.

\begin{lemma}\label{l:appidecom}
The sequence $\big\{ \Phi^* G_n \Phi \}_{n = 1}^\infty$ converges strongly to the identity operator on $X$.
\end{lemma}
\begin{proof}
Let $k \in \nn$ and let $a \in A$. Apply the notation $e_k \cd a \in H_A$ for the sequence with zeroes everywhere except for the element $a$ in position $k$.

For each $n \in \nn$, we have that
\[
\begin{split}
& (\Phi^* G_n \Phi)(\xi_k \cd a) \\
& \q = \big( \Phi^* G_n\big)( \sum_{j = 1}^\infty e_j \cd \inn{\xi_j,\xi_k} \cd a \big)
= \big( \Phi^* G_n G \big)( e_k \cd a) \\
& \q = \big( \Phi^* (G + 1/n)^{-1} G \big)( e_k \cd a)
= \Phi^*(e_k \cd a) - 1/n \cd \big( \Phi^* (G + 1/n)^{-1} \big)(e_k \cd a) \\
& \q = \xi_k \cd a - 1/n \cd \big( \Phi^* (G + 1/n)^{-1} \big)(e_k \cd a)
\end{split}
\]

Thus, in order to show that $(\Phi^* G_n \Phi)(\xi_k \cd a) \to \xi_k \cd a$ it suffices to show that 
\[
\big\| 1/n \cd \Phi^* (G + 1/n)^{-1} \big\| \to 0
\]
To this end, we simply notice that
\[
\big\| 1/n \cd \Phi^* (G + 1/n)^{-1} \big\|^2 \leq \frac{1}{n^2} \cd \big\| (G + 1/n)^{-1} \cd G \cd (G + 1/n)^{-1} \big\| \leq 1/n
\]
for all $n \in \nn$. We have thus proved that $(\Phi^* G_n \Phi)(\eta) \to \eta$ for all $\eta \in \T{span}_A\{\xi_k \, | \, k \in \nn \}$. 

Therefore, since the $A$-span of the sequence $\{\xi_k\}_{k = 1}^\infty$ is dense in $X$ it is enough to show that the sequence $\{ \Phi^* G_n \Phi \}_{n = 1}^\infty$ is bounded in operator norm. But this follows from the estimate
\[
\big\| \Phi^* G_n \Phi \big\| = \big\| G_n^{1/2} \Phi \Phi^* G_n^{1/2} \big\|
= \big\| G \cd (1/n + G)^{-1} \big\| \leq 1
\]
which is valid for all $n \in \nn$.
\end{proof}

For each $n \in \nn$ define the compact operator $\Psi_n := (G_n - G_{n - 1})^{1/2} \Phi : X \to H_A$. Remark that the difference $G_n - G_{n-1}$ is positive and invertible for all $n \in \nn$, indeed
\[
\begin{split}
G_n - G_{n-1} 
& = \big(G + 1/n \big)^{-1} - \big( G + 1/(n-1)\big)^{-1} \\
& = \big(G + 1/n \big)^{-1} \cd \frac{1}{n\cd (n - 1)} \cd \big( G + 1/(n-1)\big)^{-1}
\end{split}
\]
for all $n \geq 2$. Notice also that the adjoint of $\Psi_n : X \to H_A$ is given by $\Psi_n^* = \Phi^* \cd (G_n - G_{n - 1})^{1/2} : H_A \to X$ for all $n \in \nn$.

For each Hilbert $C^*$-module $Y$ over a $C^*$-algebra $B$, let $Y^\infty$ denote the Hilbert $C^*$-module over $B$ which consists of all sequences $\{\eta_n\}_{n = 1}^\infty$ of elements in $Y$ such that the sum $\sum_{n = 1}^\infty \inn{\eta_n,\eta_n}$ is convergent in $B$. The inner product on $Y^\infty$ is given by $\inn{\{\eta_n\},\{\ze_n\}} := \sum_{n = 1}^\infty \inn{\eta_n,\ze_n}$. The right-module structure is given by $\{\eta_n\} \cd b := \{ \eta_n \cd b \}$. For each $\eta \in Y$ and each $n \in \nn$, we denote the sequence in $Y^\infty$ with $\eta$ in position $n$ and zeroes elsewhere by $e_n \cd \eta$.

\begin{lemma}\label{l:limexi}
The sequence $\big\{ \sum_{n = 1}^N e_n \cd \Psi_n(\eta) \big\}_{N = 1}^\infty$ converges in $H_A^\infty$ for all $\eta \in X$.
\end{lemma}
\begin{proof}
Let $\eta \in X$. We need to prove that the sequence $\big\{ \sum_{n = 1}^N e_n \cd \Psi_n(\eta) \big\}_{N = 1}^\infty$ is a Cauchy sequence in $H_A^\infty$.

Thus, let $M, N \in \nn$ with $M \geq N$ be given. We may then compute as follows,
\[
\begin{split}
& \big\| \sum_{n = N + 1}^M e_n \cd \Psi_n(\eta) \big\|^2
= \big\| \sum_{n = N + 1}^M \binn{ \Psi_n(\eta),\Psi_n(\eta) } \big\| \\
& \q = \big\| \sum_{n = N + 1}^M \binn{ \eta, \Phi^* (G_n - G_{n - 1}) \Phi (\eta) } \big\|
= \big\| \binn{\eta, \Phi^* (G_M - G_N) \Phi(\eta)} \big\|
\end{split}
\]
The result of the present lemma now follows by an application of Lemma \ref{l:appidecom}.
\end{proof}

Define the $A$-linear map $\Psi : X \to H_A^\infty$, $\Psi : \eta \mapsto \sum_{n = 1}^\infty e_n \cd \Psi_n(\eta)$. Remark that it follows from Lemma \ref{l:limexi} that the sum in the definition of $\Psi$ makes sense. 

\begin{prop}\label{p:isoinc}
\[
\inn{\Psi(\xi),\Psi(\eta)} = \inn{\xi,\eta} \q \T{for all } \xi,\eta \in X
\]
\end{prop}
\begin{proof}
Let $\xi,\eta \in X$. By Lemma \ref{l:appidecom} we have that
\[
\begin{split}
\inn{\Psi(\xi),\Psi(\eta)} 
& = \sum_{n = 1}^\infty \inn{\Psi_n(\xi),\Psi_n(\eta)} 
= \sum_{n = 1}^\infty \inn{\xi,\Phi^* (G_n - G_{n-1}) \Phi (\eta)} \\
& = \lim_{N \to \infty} \inn{\xi,(\Phi^* G_N \Phi)(\eta)} 
= \inn{\xi,\eta}
\end{split}
\]
This proves the proposition.
\end{proof}

It follows from the above proposition that $\Psi : X \to H_A^\infty$ is bounded (it is in fact an isometry). To construct the adjoint, define the $A$-linear map $\Psi^* : \op_{n=1}^\infty H_A \to X$, $\Psi^* : \sum_{n = 1}^\infty e_n \cd x_n \mapsto \sum_{n = 1}^\infty \Psi_n^*(x_n)$, where $\op_{n = 1}^\infty H_A$ denotes the dense $A$-submodule in $H_A^\infty$ consisting of all \emph{finite} sequences in $H_A$. It then follows from the above proposition that
\[
\big\| \inn{\Psi^*(\sum_{n = 1}^\infty e_n \cd x_n),\xi} \big\| = \big\|  \inn{\sum_{n = 1}^\infty e_n \cd x_n, \Psi(\xi)} \big\| 
\leq \big\| \sum_{n = 1}^\infty e_n \cd x_n \| \cd \| \xi \|
\]
for all $\sum_{n = 1}^\infty e_n \cd x_n \in \op_{n = 1}^\infty H_A$ and all $\xi \in X$. This implies that $\Psi^* : \op_{n = 1}^\infty H_A \to X$ extends to a bounded $A$-linear map $\Psi^* : H_A^\infty \to X$ and it is not hard to see that this operator is the adjoint of $\Psi : X \to H_A^\infty$.

The next proposition now follows immediately from Proposition \ref{p:isoinc}.

\begin{prop}\label{p:riginv}
\[
\Psi^* \Psi = 1_X : X \to X
\]
\end{prop}

Let $\al : \nn \to \nn \ti \nn$, $\al(n) = (\al_1(n),\al_2(n))$ be a bijection. We then have an associated unitary isomorphism of Hilbert $C^*$-modules $U_\al : H_A \to H_A^\infty$ defined by
\begin{equation}\label{eq:intuni}
U_\al :  e_n \cd a \mapsto e_{\al_1(n)} \cd ( e_{\al_2(n)} \cd a)
\end{equation}

The continuous absorption theorem can now be stated and proved:

\begin{thm}\label{t:conabs}
There exists a bounded adjointable isometry $W : X \to H_A$.
\end{thm}
\begin{proof}
Define the bounded adjointable operator $W := U_\al^* \Psi : X \to H_A$. The result of the theorem then follows immediately from Proposition \ref{p:riginv}.
\end{proof}

Notice that $P := W W^* : H_A \to H_A$ is an orthogonal projection and that $W$ induces a unitary isomorphism of Hilbert $C^*$-modules $W : X \to P H_A$ where $P H_A \su H_A$ has inherited the structure of a Hilbert $C^*$-module from $H_A$.

The result of Theorem \ref{t:conabs} can be strengthened slightly. Indeed, we have the following proposition (which is non-trivial since we are in a non-unital setting):

\begin{prop}\label{p:confra}
There exists a sequence $\{\ze_k\}_{k = 1}^\infty$ of elements in $X$ such that
\[
W(\eta) = \{ \inn{\ze_k,\eta} \}_{k = 1}^\infty \q \T{for all } \eta \in X
\]
\end{prop}
\begin{proof}
It suffices to fix an $n \in \nn$ and find a sequence $\{ \nu_m\}_{m = 1}^\infty$ in $X$ such that
\[
\Psi_n(\eta) = \{ \inn{\nu_m,\eta}\}_{m = 1}^\infty \q \T{for all } \eta \in X
\]
To find the elements $\nu_m \in X$, let us also fix an $m \in \nn$ and consider the bounded adjointable operator $P_m : H_A \to A$, $P_m : \sum_{k = 1}^\infty e_k a_k \mapsto a_m$. We then have that
\[
P_m \Psi_n = P_m \sqrt{G_n - G_{n-1}} \Phi
\]
Notice now that the bounded adjointable operator $P_m \sqrt{G_n - G_{n-1}} \Phi : X \to A$ is compact (since $\Phi : X \to H_A$ is compact). As a consequence, there exists an element $\nu_m \in X$ with
\[
(P_m \sqrt{G_n - G_{n-1}} \Phi)(\eta) = \inn{\nu_m,\eta} \q \T{for all } \eta \in X
\]
This proves the proposition.
\end{proof}

\begin{remark}
The sequence $\{\ze_k\}_{k = 1}^\infty$ in $X$ which implements $W : X \to H_A$ is a ``standard normalized tight frame'' in the terminology of M. Frank and D. R. Larson, see \cite[Definition 2.1]{FrLa:FHM} (notice however that we never assume that $A$ is unital).
\end{remark}

\section{Differentiable absorption}\label{s:difabs}
Let $X$ be a countably generated Hilbert $C^*$-module over a $C^*$-algebra $A$. Furthermore, let $B$ be a $C^*$-algebra and let $\rho : A \to B$ be an injective $*$-homomorphism.

The ``differentiable structure'' on $A$ will come in the form of a dense $*$-subalgebra $\sA \su A$ and a linear map $\de : \sA \to B$ such that
\[
\de(a_1 \cd a_2) = \de(a_1) \cd \rho(a_2) + \rho(a_1) \cd \de(a_2) \q \T{and} \q
\de(a^*) = - \de(a)^*
\]
for all $a, a_1,a_2 \in \sA$. The derivation $\de : \sA \to B$ is required to be \emph{closed}. Thus, whenever $\{a_n\}$ is a sequence in $\sA$ such that $\de(a_n) \to b$ and $a_n \to 0$ for some $b \in B$ we may conclude that $b = 0$.

We let $A_\de$ denote the completion of $\sA$ with respect to the norm
\[
\| \cd \|_\de : \sA \to [0,\infty) \q \| a \|_\de := \| a \| + \| \de(a)\| 
\]
It follows by closedness that $\de : \sA \to B$ extends to a well-defined derivation $\de : A_\de \to B$. Remark that $\| a^* \|_\de = \| a \|_\de$ for all $a \in A_\de$, but that the $C^*$-identity does not hold for the norm $\| \cd \|_\de$.

The countably generated Hilbert $C^*$-module $X$ is assumed to be compatible with the differentiable structure on $A$ by the following condition: There exists a sequence $\{\xi_n\}_{n = 1}^\infty$ in $X$ such that
\[
\inn{\xi_n,\xi_m} \in \sA \q \T{for all } n,m \in \nn
\]
and such that $\T{span}_{A}\{ \xi_n \, | \, n \in \nn \}$ is dense in $X$.

Without loss of generality, we may then assume that
\begin{equation}\label{eq:norestinn}
\| \inn{\xi_n,\xi_m} \|_\de \leq \frac{1}{n^2 \cd m^2} \q \T{for all } n,m \in \nn
\end{equation}
\vspace{5pt}

\emph{The conditions stated above will remain in effect throughout this section.}
\bigskip

Let $M_\infty(\sA)$ denote the $*$-algebra of all finite matrices over $\sA$. We will think of $M_\infty(\sA)$ as a dense $*$-subalgebra of the compact operators $\sK(H_A)$ on the Hilbert $C^*$-module $H_A$. There is a unique injective $*$-homomorphism $\rho : \sK(H_A) \to \sK(H_B)$ such that $\rho( \{ a_{ij}\}) = \{ \rho(a_{ij}) \}$ for all finite matrices $\{a_{ij}\} \in M_\infty(\sA)$. Likewise, we may extend $\de : \sA \to B$ to a closed derivation $\de : M_\infty(\sA) \to \sK(H_B)$.

We will apply the notation $\sK(H_A)_\de$ for the Banach $*$-algebra obtained as the completion of $M_\infty(\sA)$ with respect to the norm $\| \cd \|_\de : a \mapsto \| a \| + \| \de(a) \|$.
%

The unitalization of $\sK(H_A)_\de$ is denoted by $\wit{\sK(H_A)_\de}$. This unital $*$-algebra becomes a unital Banach $*$-algebra when equipped with the norm $\| \cd \|_\de : \wit{\sK(H_A)_\de} \to [0,\infty)$, $\| (a,\la) \|_\de := \| a + \la \| + \|\de(a)\|$. Here we are thinking of $a + \la$ as a bounded adjointable operator on the standard module $H_A$. Notice that our $*$-homomorphism $\rho : \sK(H_A) \to \sK(H_B)$ can be extended uniquely to a unital $*$-homomorphism $\rho : \wit{\sK(H_A)} \to \sL(H_B)$ and that our derivation $\de : M_\infty(\sA) \to \sK(H_B)$ can be extended uniquely to a closed derivation $\de : \wit{\sK(H_A)_\de} \to \sL(H_B)$ such that $\de\big( (0,\la)\big) = 0$ for all $\la \in \cc$.

We are now ready to prove the first result of this section:

\begin{lemma}
The sequence of finite matrices $\big\{ \{\inn{\xi_n,\xi_m}\}_{n,m = 1}^N\big\}_{N = 1}^\infty$ converges to an element $G \in \wit{\sK(H_A)_\de}$ with positive spectrum.
\end{lemma}
\begin{proof}
We first remark that $\big\{ \inn{\xi_n,\xi_m}\big\}_{n,m = 1}^N$ determines a positive element in the $C^*$-algebra $M_N(A)$ for all $N \in \nn$.

Next, we notice that the spectrum of an element $a$ in the unital Banach algebra $\wit{M_N(A_\de)}$ agrees with the spectrum of $a$ as an element in the unital $C^*$-algebra $\wit{M_N(A)}$. This is a consequence of spectral invariance, see \cite[Proposition 3.12]{BlCu:DNS}.

These observations imply that $\big\{ \inn{\xi_n,\xi_m} \big\}_{n,m = 1}^N \in \wit{M_N(A_\de)}$ has positive spectrum for all $N \in \nn$. It is therefore enough to show that the sequence $\big\{ \{\inn{\xi_n,\xi_m} \}_{n,m = 1}^N \big\}_{N = 1}^\infty$ is Cauchy in $\wit{\sK(H_A)_\de}$.

To this end, let $N,M \in \nn$ with $M \geq N$ be given and notice that
\[
\begin{split}
& \big\| \big\{\inn{\xi_n,\xi_m}\big\}_{n,m = 1}^M - \big\{\inn{\xi_n,\xi_m}\big\}_{n,m = 1}^N\big\|_\de \\
& \q \leq \sum_{n = N + 1}^M \sum_{m = 1}^M \| \inn{\xi_n,\xi_m} \|_\de + \sum_{n = 1}^N \sum_{m = N + 1}^M \| \inn{\xi_n,\xi_m} \|_\de \\
& \q \leq 2 \cd \sum_{m = 1}^\infty \frac{1}{m^2} \cd \sum_{n = N + 1}^M \frac{1}{n^2}
\end{split}
\]
where the last inequality follows by \eqref{eq:norestinn}. This shows that the sequence $\big\{ \{\inn{\xi_n,\xi_m}\}_{n,m = 1}^N \big\}_{N = 1}^\infty$ is Cauchy in $\wit{\sK(H_A)_\de}$.
\end{proof}

For each $n \in \nn$, we define the element
\[
\begin{split}
H_n & := (1/n + G)^{-1} - (1/(n - 1) + G)^{-1} \\
& = (1 + n \cd G)^{-1} \cd (1 + (n-1) \cd G)^{-1}
\end{split}
\]
in $\wit{\sK(H_A)_\de}$, where $H_1 := (1 + G)^{-1}$. Since the spectrum of $H_n$ is strictly positive, it has a well-defined square root in $\wit{\sK(H_A)_\de}$,
\[
\sqrt{H_n} = (1 + n \cd G)^{-1/2} \cd (1 + (n-1) \cd G)^{-1/2}
\]

\begin{lemma}\label{l:dersqr}
We have the expression
\[
\begin{split}
& \de\big( (1 + n G)^{-1/2} \big) \\
& \q = - \frac{n}{\pi} \cd \int_0^\infty \la^{-1/2} \cd \rho\big( (1 + \la + n G)^{-1} \big) \cd \de(G) \cd \rho\big( (1 + \la + n \cd G)^{-1}\big) \, d\la
\end{split}
\]
where the integral converges in the operator norm on $\sL(H_B)$.
\end{lemma}
\begin{proof}
The element $(1 + n G)^{-1/2} \in \wit{\sK(H_A)_\de}$ can be rewritten as the integral
\[
\frac{1}{\pi} \cd \int_0^\infty \la^{-1/2} \cd (1 + \la + n \cd G)^{-1} \, d\la
\]
which converges absolutely in the norm $\| \cd \|_\de : \wit{\sK(H_A)_\de} \to [0,\infty)$. It is therefore enough to check that
\[
\de\big( (1 + \la + n \cd G)^{-1} \big) = - \rho \big( (1 + \la + n G)^{-1} \big) \cd  n \cd \de(G) \cd \rho\big( (1 + \la + n \cd G)^{-1} \big)
\]
But this follows from a standard computation, using that $\de : \wit{\sK(H_A)_\de} \to \sL(H_B)$ is a derivation with respect to $\rho : \wit{\sK(H_A)} \to \sL(H_B)$.
\end{proof}

The estimate in the following lemma is of central importance for the differentiable absorption theorem.

\begin{lemma}\label{l:delest}
Let $\ep \in (0,1/2)$. There exists a constant $C_\ep > 0$ such that
\[
\big\| \de(\sqrt{H_n} \cd G^2) \big\| \leq C_\ep \cd \frac{1}{n^{1-\ep}}
\]
for all $n \in \nn$.
\end{lemma}
\begin{proof}
Let $n \geq 2$. Using that $\de : \wit{\sK(H_A)_\de} \to \sL(H_B)$ is a derivation we obtain that
\begin{equation}\label{eq:delsum}
\begin{split}
\de(\sqrt{H_n} \cd G^2) & = \de(G) \cd \sqrt{H_n} \cd G + G \cd \sqrt{H_n} \cd \de(G) \\ 
& \qq + G \cd \de\big( (1 + n G)^{-1/2} \big) \cd ( 1 + (n-1) G )^{-1/2} \cd G \\
& \qq + G \cd (1 + n G)^{-1/2} \cd \de\big( ( 1 + (n-1) G )^{-1/2}\big) \cd G
\end{split}
\end{equation}
where we have suppressed the unital $*$-homomorphism $\rho : \wit{\sK(H_A)} \to \sL(H_B)$.

Now, since $G \in \sK(H_A)_\de$ determines a positive element in the unital $C^*$-algebra $\wit{\sK(H_A)}$, we have that
\[
\| G \cd (1 + \la + n G)^{-1} \| \leq \frac{1}{n}
\]
for all $\la \geq 0$.

Using the above estimate we obtain the following inequalities
\[
\begin{split}
& \big\| \de(G) \cd \sqrt{H_n} \cd G + G \cd \sqrt{H_n} \cd \de(G) \big\| \\
& \q \leq 2 \cd \| \de(G) \| \cd \big\| (1 + (n-1) G)^{-1/2} G^{1/2} \big\| \cd \big\| (1 + n G)^{-1/2} G^{1/2} \big\| \\
& \q \leq 2 \cd \| \de(G) \| \cd \frac{1}{\sqrt{n} \cd \sqrt{n-1}}
\end{split}
\]

To continue, we apply Lemma \ref{l:dersqr} to compute as follows,
\[
\begin{split}
& G \cd \de\big( (1 + n G)^{-1/2} \big) \cd ( 1 + (n-1) G )^{-1/2} \cd G \\
& \q = - \frac{1}{\pi} \cd \int_0^\infty \la^{-1/2} \cd (n G) \cd (1 + \la + n G)^{-1} \cd \de(G) \\ 
& \qqqq \cd G^{1/2 - \ep} \cd (1 + \la + n G)^{-1} \, d\la \\ 
& \qq \qqq \cd G^{1/2 + \ep} \cd (1 + (n-1) G)^{-1/2}
\end{split}
\]
As a consequence, we obtain that
\[
\begin{split}
& \big\| G \cd \de\big( (1 + n G)^{-1/2} \big) \cd ( 1 + (n-1) G )^{-1/2} \cd G \big\| \\
& \q \leq \frac{1}{\pi} \cd \int_0^\infty \la^{-1/2} \cd \| \de(G) \| \cd (1 + \la)^{-1/2 - \ep} \cd \| G^{1/2 - \ep} \cd (1 + \la + n G)^{-1/2 + \ep} \| \, d\la \\ 
& \qqq \cd \| G^\ep \| \cd \frac{1}{\sqrt{n-1}} \\
& \q \leq \| \de(G) \| \cd \| G^\ep \| \cd \frac{1}{(n-1)^{1/2} \cd n^{1/2 - \ep}\cd \pi} \cd \int_0^\infty \la^{-1/2} (1 + \la)^{-1/2 - \ep} \, d\la 
\end{split}
\]

A similar computation shows that
\[
\begin{split}
& \big\| G \cd (1 + n G)^{-1/2} \cd \de \big( ( 1 + (n-1) G )^{-1/2} \big) \cd G \big\| \\
& \q \leq \| \de(G) \| \cd \| G^\ep \| \cd \frac{1}{(n-1)^{1/2-\ep} \cd n^{1/2}\cd \pi} \cd \int_0^\infty \la^{-1/2} (1 + \la)^{-1/2 - \ep} \, d\la
\end{split}
\]

A combination of all the above estimates and the identity in \eqref{eq:delsum} proves the claim of the proposition.
\end{proof}

Recall from Section \ref{s:conabs} that the compact operators $\Phi^* : H_A \to X$ and $\Phi : X \to H_A$ are defined by $\Phi^* : \{a_k\}_{k = 1}^\infty \mapsto \sum_{k = 1}^\infty \xi_k \cd a_k$ and $\Phi : \eta \mapsto \{ \inn{\xi_k,\eta}\}_{k = 1}^\infty$.

Furthermore, for each $n \in \nn$, we have the compact operators $\Psi_n := \sqrt{H_n} \Phi : X \to H_A$ and $\Psi_n^* := \Phi^* \sqrt{H_n} : H_A \to X$.

Finally, for each $N \in \nn$ we have the compact operators $V_N : X \to H_A^\infty$ and $V_N^* : H_A^\infty \to X$ defined by $V_N : \eta \mapsto \{ \Psi_n(\eta) \}_{n = 1}^N$ and $V_N^* : \{ x_n\}_{n = 1}^\infty \mapsto \sum_{n = 1}^N \Psi^*_n(x_n)$. It was proved in Section \ref{s:conabs} that the sequence $\{V_N\}_{N = 1}^\infty$ converges strongly to a bounded adjointable isometry $\Psi : X \to H_A^\infty$. The adjoint of $\Psi$ is given by $\Psi^* : \sum_{n = 1}^\infty e_n \cd x_n \mapsto \sum_{n = 1}^\infty \Psi_n^*(x_n)$.

For each $N \in \nn$ we define the compact operator
\[
\de\big( \T{diag}(G) V_N \Phi^*) \in \sK(H_B,H_B^\infty) \q \de\big( \T{diag}(G) V_N \Phi^* \big) : x \mapsto \sum_{n = 1}^N e_n \cd \de(G^2 \sqrt{H_n})(x)
\]
where $\T{diag}(G) : H_A^\infty \to H_A^\infty$ refers to the (non-compact) diagonal operator $\T{diag}(G) : \sum_{n = 1}^\infty e_n x_n \mapsto \sum_{n = 1}^\infty e_n G(x_n)$ induced by the (compact operator) $G : H_A \to H_A$.

We note the following consequence of the above Lemma \ref{l:delest}:

\begin{lemma}\label{l:caudercom}
The sequence of compact operators $\{ \de\big( \T{diag}(G) V_N \Phi^* \big) \}_{N = 1}^\infty$ is a Cauchy sequence in $\sK(H_B,H_B^\infty)$. 
\end{lemma}
\begin{proof}
By Lemma \ref{l:delest} we may choose a constant $C > 0$ such that
\[
\begin{split}
& \| \de\big( \T{diag}(G) V_N \Phi^* \big)(x) - \de\big( \T{diag}(G) V_M \Phi^* \big)(x) \|^2 = \| \sum_{n = N+1}^M e_n \de(G^2 \sqrt{H_n} )(x) \|^2 \\
& \q = \big\| \sum_{n = N+1}^M \binn{ \de(G^2 \sqrt{H_n})x, \de(G^2 \sqrt{H_n})x} \big\|
\leq C \sum_{n = N + 1}^M \frac{1}{n^{3/2}} \| x\|^2 
\end{split}
\]
for all $N,M \in \nn$ with $M \geq N$ and all $x \in H_B$. This proves the lemma. 
\end{proof}

The next lemma is a consequence of Lemma \ref{l:limexi}.

\begin{lemma}\label{l:seqcon}
The sequence of compact operators $\{ V_N \Phi^* \}_{N = 1}^\infty$ converges in operator norm to $\Psi \Phi^* : H_A \to H_A^\infty$.
\end{lemma}
\begin{proof}
This follows since $\Phi : X \to H_A$ (and hence $\Phi^* : H_A \to X$) is compact and since the bounded sequence $\{ V_N\}_{N = 1}^\infty$ converges strongly to $\Psi : X \to H_A^\infty$.
\end{proof}

\begin{prop}\label{p:seqpro}
The sequence $\{ \T{diag}(G) V_N V_N^* \}_{N = 1}^\infty$ in $\sK(H_A^\infty)$ converges in operator norm to $\T{diag}(G) \Psi \Psi^* : H_A^\infty \to H_A^\infty$.
\end{prop}
\begin{proof}
Let $N \in \nn$ and remark that 
\[
\{ \T{diag}(G) V_N V_N^*\}_{n,m} = G^2 \sqrt{H_m} \cd \sqrt{H_n} = \sqrt{H_m} \Phi \Phi^* \Phi \Phi^* \sqrt{H_n}
\]
for all $n,m \in \{1,\ldots,N\}$. It follows that $\T{diag}(G) V_N V_N^* = V_N \Phi^* \Phi V_N^*$. The result of the proposition is now a consequence of Lemma \ref{l:seqcon}.
%
\end{proof}

In order to formulate our next result we reiterate the construction of the Banach $*$-algebra $\sK(H_A)_\de$. Indeed, we may consider the finite matrices $M_\infty\big( \sK(H_A)_\de \big)$ as a dense $*$-subalgebra of the compact operators $\sK(H_A^\infty)$ on the standard module $H_A^\infty$. The $*$-homomorphism $\rho : \sK(H_A) \to \sK(H_B)$ can then be extended uniquely to a $*$-homomorphism $\rho : \sK(H_A^\infty) \to \sK(H_B^\infty)$ such that $\rho \{ x_{ij}\} = \{ \rho(x_{ij})\}$ for all $\{ x_{ij}\} \in M_\infty( \sK(H_A))$. Likewise, we may extend $\de$ uniquely to a closed derivation $\de : M_\infty\big( \sK(H_A)_\de \big) \to \sK(H_B^\infty)$ such that $\de\{ x_{ij} \} := \{ \de(x_{ij}) \}$. We denote the Banach $*$-algebra defined as the completion of $M_\infty\big( \sK(H_A)_\de \big)$ with respect to the norm $\| \cd \|_\de : x \mapsto \| x\| + \| \de(x) \|$ by $\sK(H_A^\infty)_\de$.

We note that we have an isometric isomorphism of Banach $*$-algebras $\sK(H_A^\infty)_\de \to \sK(H_A)_\de$ defined by conjugasion with the unitary operator $U_\al : H_A \to H_A^\infty$ introduced in \eqref{eq:intuni}.

\begin{prop}\label{p:comdelcau}
The sequence $\{ \T{diag}(G)^2 V_N V_N^* \}_{N = 1}^\infty$ in $M_\infty(\sK(H_A)_\de)$ is Cauchy in $\sK(H_A^\infty)_\de$. 
\end{prop}
\begin{proof}
We know from Proposition \ref{p:seqpro} that $\T{diag}(G)^2 V_N V_N^*$ converges to $\T{diag}(G)^2 \Psi \Psi^*$ in $\sK(H_A^\infty)$. It is therefore enough to show that $\big\{ \de\big( \T{diag}(G)^2 V_N V_N^*\big) \big\}_{N = 1}^\infty$ is a Cauchy sequence in $\sK(H_B^\infty)$.

Let now $N \in \nn$ and notice that
\[
(\T{diag}(G) V_N \Phi^*)(x) = \sum_{n = 1}^N e_n \cd (G \sqrt{H_n} G)(x) = \sum_{n = 1}^N e_n \cd (\sqrt{H_n} \Phi \Phi^* G)(x)
= (V_N \Phi^* G)(x)
\]
for all $x \in H_A$. We thus have that $\T{diag}(G) V_N \Phi^* = V_N \Phi^* G$.

We may therefore compute as follows,
\[
\begin{split}
& \de\big( \T{diag}(G)^2 V_N V_N^*\big) \\
& \q = \de\big( \T{diag}(G) V_N \Phi^* \Phi V_N^* \big) 
= \de\big( \T{diag}(G) V_N \Phi^*\big) \Phi V_N^* + \T{diag}(G) V_N \Phi^* \de(\Phi V_N^*) \\
& \q = \de\big( \T{diag}(G) V_N \Phi^* \big) \Phi V_N^* + V_N \Phi^* \de(G \Phi V_N^*) - V_N \Phi^* \de(G) \Phi V_N^* \\
& \q = \de\big( \T{diag}(G) V_N \Phi^* \big) \Phi V_N^* - V_N \Phi^* \de\big( \T{diag}(G) V_N \Phi^* \big)^* - V_N \Phi^* \de(G) \Phi V_N^*
\end{split}
\]
The result of the proposition now follows by Lemma \ref{l:seqcon} and Lemma \ref{l:caudercom}.
\end{proof}

\begin{lemma}\label{l:imaden}
The image of $ \Psi^* \T{diag}(G) \Psi : X \to X$ is dense in $X$ and $\T{diag}(G) \Psi \Psi^* = \Psi \Psi^* \T{diag}(G)$.
\end{lemma}
\begin{proof}
By Proposition \ref{p:seqpro} we know that $\T{diag}(G) \Psi \Psi^* = \lim_{N \to \infty} \T{diag}(G) V_N V_N^*$ and that $\Psi \Psi^* \T{diag}(G) = \lim_{N \to \infty} V_N V_N^* \T{diag}(G)$. To show that $\T{diag}(G) \Psi \Psi^* = \Psi \Psi^* \T{diag}(G)$ is therefore suffices to show that $V_N V_N^* \T{diag}(G) = \T{diag}(G) V_N V_N^*$ for all $N \in \nn$. But this follows by noting that
\[
\big(V_N V_N^* \T{diag}(G)\big)_{n,m} = \sqrt{H_n} G \sqrt{H_m} G = G \sqrt{H_n} G \sqrt{H_m} = \big(\T{diag}(G) V_N V_N^*\big)_{n,m}
\]
for all $N \in \nn$ and all $n,m \in \{1,\ldots,N\}$.

In order to prove that the image of $\Psi^* \T{diag}(G) \Psi : X \to X$ is dense we note that
\[
\begin{split}
& \T{span}_A\big\{ \xi \in \T{Im}(\Phi^* G (G + 1/n)^{-1}) \, | \, n \in \nn \big\}
\su \T{span}_A\big\{ \xi \in  \T{Im}(\Phi^* G \sqrt{H_n}) \, | \, n \in \nn \big\} \\
& \q \su \T{Im}\big(\Psi^* \T{diag}(G) \big) = \T{Im}\big(\Psi^* \T{diag}(G) \Psi \Psi^* \big) \su \T{Im}\big(\Psi^* \T{diag}(G) \Psi\big)
\end{split}
\]

Since the image of $\Phi^* : H_A \to X$ is dense by the standing conditions on our Hilbert $C^*$-module $X$ it therefore suffices to show that the sequence $\{ \Phi^* G (1/n + G)^{-1} \}_{n = 1}^\infty$ of bounded adjointable operators converges in operator norm to $\Phi^* : H_A \to X$. But this follows since
\[
\frac{1}{n} \| \Phi^* (1/n + G)^{-1} \| \leq \frac{1}{\sqrt{n}} 
\]
for all $n \in \nn$. See the proof of Lemma \ref{l:appidecom}.
\end{proof}


We are now ready to prove the differentiable absorption theorem. This is the first main result of the present paper.

\begin{thm}\label{t:difabs}
There exists a bounded adjointable isometry $W : X \to H_A$ and a positive selfadjoint bounded operator $K : H_A \to H_A$ such that
\begin{enumerate}
\item $K P = P K$.
\item $W^* K W : X \to X$ has dense image.
\item $P K \in \sK(H_A)$.
\item $P K^2 \in \sK(H_A)_\de$.
\end{enumerate}
where $P := W W^* : H_A \to H_A$ is the associated orthogonal projection.
\end{thm}
\begin{proof}
Let $U_\al : H_A \to H_A^\infty$ denote the unitary operator introduced in \eqref{eq:intuni}. The bounded adjointable operator $W := U_\al^* \Psi : X \to H_A$ is then an isometry. Furthermore, define the positive selfadjoint bounded operator $K := U_\al^* \T{diag}(G) U_\al : H_A \to H_A$. The result of the theorem then follows by Lemma \ref{l:imaden}, Proposition \ref{p:seqpro}, and Proposition \ref{p:comdelcau}.
\end{proof}

\begin{remark}\label{r:diffra}
As in Proposition \ref{p:confra}, we may find a sequence $\{ \ze_k \}_{k = 1}^\infty$ of elements in $X$ which implements the isometry $W : X \to H_A$ in the sense that
\[
W(\eta) = \{ \inn{\ze_k,\eta} \}_{k = 1}^\infty \q \T{for all } \eta \in X 
\]
\end{remark}

\section{Gra\ss mann connections}\label{s:gracon}
\emph{Throughout this section we will work in the setting outlined in the beginning of Section \ref{s:difabs}. We will then let $W : X \to H_A$ and $K : H_A \to H_A$ be fixed bounded adjointable operators which satisfy the properties stated in Theorem \ref{t:difabs}. Furthermore, we let $\{\ze_k\}_{k = 1}^\infty$ be a sequence in $X$ which implements $W$, see Remark \ref{r:diffra}}.

We shall in this section see how to construct a dense $A_\de$-submodule of $\sX \su X$ together with a Hermitian $\de$-connection on $\sX$.

In order to construct $\sX$ we recall the following, see \cite[Definition 3.3]{KaLe:SFU} and \cite[Page 119]{Mes:UCN}:

\begin{dfn}
The \emph{standard module} over $A_{\de}$ consists of all sequences $\{a_n\}_{n = 1}^\infty$ of elements in $A_\de$ such that
\[
\{a_n\} \in H_A \q \T{and} \q \{\de(a_n) \} \in H_B
\]
The standard module over $A_\de$ is denoted by $H_{A_\de}$.
\end{dfn}

The standard module $H_{A_\de}$ is a dense $A_\de$-submodule of the standard module $H_A$. Furthermore, it was proved in \cite[Page 505]{KaLe:SFU} that
\[
\inn{x,y} \in A_\de \q \T{for all } x,y \in H_{A_\de}
\]
where $\inn{\cd, \cd} : H_A \ti H_A \to A$ denotes the inner product on $H_A$.

The standard module becomes a Banach space when equipped with the norm
\[
\| \cd \|_\de : \{a_n \} \mapsto \| \{a_n\} \| + \| \{\de(a_n) \} \|
\]

Each element $T \in \sK(H_A)_\de \su \sK(H_A)$ restricts to a bounded operator $T : H_{A_\de} \to H_{A_\de}$. Indeed, the map
\[
M_\infty(A_\de) \ti H_{A_\de} \to H_{A_\de} \q \big( \{a_{ij} \}, \{b_n\} \big) \mapsto \{ \sum_{n = 1}^\infty a_{in} \cd b_n \} 
\]
satisfies the inequality $\| A \cd b \|_\de \leq \| A\|_\de \cd \| b\|_\de$ for all $A \in M_\infty(A_\de)$ and $b \in H_{A_\de}$.

We may now define the $A_\de$-submodule $\sX \su X$ as the following image:
\begin{equation}\label{eq:defc1m}
\sX := \T{Im}\big( W^* K^2 : H_{A_\de} \to X \big)
\end{equation}
The properties of $\sX$ are summarized in the next lemma:

\begin{lemma}\label{l:proc1m}
The $A_\de$-submodule $\sX \su X$ is dense. Furthermore, $W(\xi) \in H_{A_\de}$ and $\inn{\xi,\eta} \in A_\de$ for all $\xi,\eta \in \sX$.
\end{lemma}
\begin{proof}
To see that $\sX \su X$ is dense, recall from Theorem \ref{t:difabs} that $W^* K W : X \to X$ has dense image. It follows that
\[
W^* K^2 W = W^* K W W^* K W : X \to X
\]
has dense image as well. In particular, we obtain that $W^* K^2 : H_A \to X$ has dense image, thus the density of $\sX \su X$ follows since $H_{A_\de} \su H_A$ is dense.

Consider now $\xi = (W^* K^2)(x)$ with $x \in H_{A_\de}$. Then $W(\xi) = (W W^* K^2)(x)$. But $W W^* K^2 \in \sK(H_A)_\de$ by Theorem \ref{t:difabs} and therefore $(W W^* K^2)(x) \in H_{A_\de}$ by the observations preceding this lemma. This proves the second claim of the present lemma.

Finally, let $\xi,\eta \in \sX$. Since $W : X \to H_A$ is an isometry, we obtain that $\inn{\xi,\eta} = \inn{W \xi , W \eta}$. But $\inn{W \xi, W \eta} \in A_\de$ since $W \xi, W \eta \in H_{A_\de}$.
\end{proof}

In order to construct the Hermitian $\de$-connection we recall the following concepts:

\begin{dfn}
The $C^*$-algebra of \emph{continuous $\de$-forms} is the smallest $C^*$-subalgebra of $B$ which contains $\rho(a_0)$ and $\de(a_1)$ for all $a_0, a_1 \in A_\de$. This $C^*$-algebra is denoted by $\Om_\de(A)$.
\end{dfn}

We remark that $\Om_\de(A)$ can be viewed as a Hilbert $C^*$-module over $\Om_\de(A)$ in the usual way (this holds for any $C^*$-algebra). Furthermore, we have an injective $*$-homomorphism $\rho : A \to \sL(\Om_\de(A))$ given by $\rho(a)(\om) = \rho(a) \cd \om$ for all $a \in A$ and $\om \in \Om_\de(A)$.

\begin{dfn}
The Hilbert $C^*$-module of \emph{continuous $X$-valued $\de$-forms} is the interior tensor product $X \hot_A \Om_\de(A)$.
\end{dfn}

Define the bounded operator $W \ot 1 : X \hot_A \Om_\de(A) \to H_{\Om_\de(A)}$, $\xi \hot \om \mapsto W(\xi)\cd \om$. Remark that it is non-obvious that $W \ot 1$ is adjointable since we do not assume that the left action of $A$ on $\Om_\de(A)$ is essential. This is none-the-less the case. Indeed, it suffices to recall that $W : X \to H_A$ is implemented by the sequence $\{ \ze_k\}_{k = 1}^\infty$ of elements in $X$. We state the result as a lemma:

\begin{lemma}\label{l:adjexi}
The bounded operator $W \ot 1 : X \hot_A \Om_\de(A) \to H_{\Om_\de(A)}$ is adjointable with adjoint $W^* \ot 1 : H_{\Om_\de(A)} \to X \hot_A \Om_\de(A)$ induced by
\[
W^* \ot 1 : \sum_{k = 1}^N e_k \cd \om_k \mapsto \sum_{k = 1}^N \ze_k \ot \om_k
\]
for all finite sequences $\sum_{k = 1}^N e_k \cd \om_k$ in $H_{\Om_\de(A)}$.
\end{lemma}

We are now in position to define our Hermitian $\de$-connection:

\begin{dfn}\label{d:gracon}
The \emph{Gra\ss mann $\de$-connection} on $\sX$ is defined by
\[
\Na_\de : \sX \to X \hot_A \Om_\de(A) \q \Na_\de := (W^* \ot 1) \de W
\]
where $\de : H_{A_\de} \to H_{\Om_\de(A)}$ is given by $\{a_n\}_{n = 1}^\infty \mapsto \{\de(a_n)\}_{n = 1}^\infty$.
\end{dfn}

The Gra\ss mann $\de$-connection can also be expressed by the formula
\[
\Na_\de : \eta \mapsto \sum_{k = 1}^\infty \ze_k \ot \de(\inn{\ze_k,\eta}) \q \forall \eta \in \sX
\]
where the sum converges in the norm on $X \hot_A \Om_\de(A)$.

We shall soon see that the Gra\ss mann $\de$-connection satisfies the Leibniz rule and is Hermitian. But we need a preliminary observation:

Observe that each element $\eta \in X$ defines a bounded adjointable operator $T_\eta : \Om_\de(A) \to X \hot_A \Om_\de(A)$, $T_\eta : \om \mapsto \eta \ot \om$. The adjoint is given by $T_\eta^* : X \hot_A \Om_\de(A) \to \Om_\de(A)$, $T_\eta^* : \xi \ot \om \mapsto \inn{\eta,\xi} \cd \om$.

\begin{thm}
The Gra\ss mann $\de$-connection $\Na_\de : \sX \to X \hot_A \Om_\de(A)$ is Hermitian and satisfies the Leibniz rule. Thus,
\begin{enumerate}
\item $\de(\inn{\xi,\eta}) = T_{\xi}^* \Na_\de(\eta) - \big( T_\eta^* \Na_\de(\xi) \big)^*$ for all $\xi,\eta \in \sX$.
\item $\Na_\de(\eta \cd a) = \Na_\de(\eta) \cd \rho(a) + \eta \ot \de(a)$ for all $\eta \in \sX$ and $a \in A_\de$.
\end{enumerate}
\end{thm}
\begin{proof}
Let $\xi,\eta \in \sX$ with $W \xi = \{a_n\}_{n = 1}^\infty$ and $W \eta = \{ b_n\}_{n = 1}^\infty$. To prove the first claim, we compute as follows:
\[
\begin{split}
\de(\inn{\xi,\eta}) & = \de\big( \sum_{n = 1}^\infty a_n^* b_n \big)
= \sum_{n = 1}^\infty \big( a_n^* \cd \de(b_n) - \de(a_n)^* \cd b_n \big) \\
& = \inn{W \xi, \de(W \eta)} - \big( \sum_{n = 1}^\infty b_n^* \cd \de(a_n) \big)^* \\
& = T_\xi^* (W^* \ot 1)\de(W \eta) - \inn{W \eta, \de(W \xi) }^* \\
& = T_\xi^* \Na_\de(\eta) - \big( T_\eta^* \Na_\de(\xi) \big)^*
\end{split}
\]
Notice that we have suppressed the injective $*$-homomorphism $\rho : A \to B$ in the above computation.

Let now $\eta \in \sX$ and $a \in A_\de$. To prove the second claim, we compute as follows:
\[
\begin{split}
\Na_\de(\eta \cd a) 
& = (W^* \ot 1) \de W(\eta \cd a)
= (W^* \ot 1) \big( (\de W)(\eta) \cd a \big) + (W^* \ot 1) \big( W(\eta) \cd \de(a) \big) \\
& = \Na_\de(\eta) \cd a + \eta \ot \de(a)
\end{split}
\]

These two computations prove the theorem.
\end{proof}

\section{Symmetric lifts of unbounded operators}\label{s:symsel}
In this section we will work in the following more refined situation:

Let $Y$ be a Hilbert $C^*$-module over a $C^*$-algebra $B$ and let $D : \sD(D) \to Y$ be an unbounded selfadjoint and regular operator. We recall that the conditions of selfadjointness and regularity are equivalent to the following two conditions:
\begin{enumerate}
\item The unbounded operator $D : \sD(D) \to Y$ is symmetric.
\item The unbounded operators $D \pm i : \sD(D) \to Y$ are surjective.  
\end{enumerate}
See \cite[Proposition 10.6]{Lan:HCM}.

Let $X$ be a Hilbert $C^*$-module over a $C^*$-algebra $A$ and suppose that $\rho : A \to \sL(Y)$ is an injective $*$-homomorphism. Suppose furthermore that we have a dense $*$-subalgebra $\sA \su A$ such that
\begin{enumerate}
\item $\rho(x) \xi \in \sD(D)$ for all $x \in \sA$ and $\xi \in \sD(D)$ and $[D,\rho(x)] : \sD(D) \to Y$ extends to a bounded adjointable operator $\de(x)$ for all $x \in \sA$.
\item There exists a sequence $\{\xi_n\}_{n = 1}^\infty$ in $X$ which generates $X$ as a Hilbert $C^*$-module and for which
\[
\inn{\xi_n,\xi_m} \in \sA \q \T{for all } n.m \in \nn
\]
\end{enumerate}

Remark that $\de(x^*) = - \de(x)^*$ since $D : \sD(D) \to Y$ is selfadjoint.
\bigskip

We let $W : X \to H_A$ and $K : H_A \to H_A$ be as in Theorem \ref{t:difabs}. Furthermore, we choose a sequence $\{\ze_k\}_{k = 1}^\infty$ in $X$ such that
\[
W(\eta) = \{ \inn{\ze_k,\eta} \}_{k = 1}^\infty \q \T{for all } \eta \in X
\]

Let $X \hot_A Y$ denote the interior tensor product of $X$ and $Y$ over $A$. Define the bounded adjointable operator $W \ot 1 : X \hot_A Y \to Y^\infty$, $W \ot 1 : \xi \ot \eta \mapsto \{ \rho(\inn{\ze_k,\xi})(\eta)\}_{k = 1}^\infty$. The adjoint of $W \ot 1$ is given by $W^* \ot 1 : Y^\infty \to X \hot_A Y$, $W^* \ot 1 : \{ \eta_k\}_{k = 1}^\infty \mapsto \sum_{k = 1}^\infty \ze_k \ot \eta_k$, where the sum converges in the norm-topology on $X \hot_A Y$, see Lemma \ref{l:adjexi}. We remark that $W \ot 1 : X \hot_A Y \to Y^\infty$ is an isometry in the sense that $(W^* \ot 1)(W \ot 1) = 1_{X \hot_A Y}$.

Define the unbounded operator $\T{diag}(D) : \sD(\T{diag}(D)) \to Y^\infty$ by $\T{diag}(D) : \{\eta_k\} \mapsto \{ D \eta_k\}$, where the domain is given by 
\[
\sD(\T{diag}(D)) := \big\{ \{\eta_k\} \in Y^\infty \, | \, \eta_k \in \sD(D) \T{ and } \{ D \eta_k\} \in Y^\infty \big\}
\]
The unbounded operator $\T{diag}(D)$ is then again selfadjoint and regular, indeed we have that $(\T{diag}(D) \pm i)^{-1} : \{ \eta_k\} \mapsto \{ (D \pm i)^{-1} \eta_k \}$ for all $\{\eta_k\} \in Y^\infty$.

Define the right $B$-submodule $\sD( 1 \ot_{\Na} D ) \su X \hot_A Y$ by
\[
\sD(1 \ot_{\Na} D) := \big\{ \si \in X \hot_A Y \, | \, (W \ot 1)(\si) \in \sD(\T{diag}(D)) \big\}
\]

\begin{lemma}\label{l:densymlif}
$\sD(1 \ot_{\Na} D)$ is dense in $X \hot_A Y$.
\end{lemma}
\begin{proof}
Let $\sX \su X$ be as in \eqref{eq:defc1m} and let $\sZ \su X \hot_A Y$ denote the image of the algebraic tensor product $\sX \ot_{A_\de} \sD(D)$ in $X \hot_A Y$. Remark that $\sZ \su X \hot_A Y$ is dense since $\sX \su X$ is dense and $\sD(D) \su Y$ is dense. It is therefore enough to show that $(W \ot 1)(\xi \ot \eta) \in \sD(\T{diag}(D))$ for all $\xi \in \sX$ and $\eta \in \sD(D)$.

Let thus $\xi \in \sX$ and $\eta \in \sD(D)$. We first remark that $\rho(\inn{\ze_k,\xi})(\eta) \in \sD(D)$ for all $k \in \nn$ since $\inn{\ze_k,\xi} \in A_\de$. It thus suffices to prove that $\big\{ D\big(\rho(\inn{\ze_k,\xi}) \eta \big) \big\} \in Y^\infty$.

However, we have that
\[
\begin{split}
\big\{ D\big(\rho(\inn{\ze_k,\xi}) \eta \big) \big\}_{k = 1}^\infty
& = \big\{ \de(\inn{\ze_k,\xi}) \eta \big\}_{k = 1}^\infty + \big\{ \rho(\inn{\ze_k,\xi}) D \eta \big\}_{k = 1}^\infty \\
& = \big\{ \de(\inn{\ze_k,\xi}) \eta \big\}_{k = 1}^\infty + (W \ot 1)(\xi \ot D\eta) \\
& = \de(W \xi)(\eta) + (W \ot 1)(\xi \ot D\eta)
\end{split}
\]
We therefore only need to show that $\de(W \xi)(\eta) \in Y^\infty$.

However, by Lemma \ref{l:proc1m} we have that $\de(W \xi) \in \sL(Y)^\infty$ for all $\xi \in \sX$. This implies the result of the lemma since each $\{ T_k\}_{k = 1}^\infty \in \sL(Y)^\infty$ yields a bounded adjointable operator $Y \to Y^\infty$, $\eta \mapsto \{ T_k \eta\}_{k = 1}^\infty$.
\end{proof}

The above lemma allows us to define the following unbounded operator
\[
1 \ot_{\Na} D := (W^* \ot 1) \T{diag}(D) (W \ot 1) : \sD( 1 \ot_{\Na} D) \to X \hot_A Y
\]
which we refer to as the \emph{symmetric lift} of $D$ with respect to the Gra\ss mann $\de$-connection $\Na$.

\begin{prop}
The unbounded operator
\[
1 \ot_{\Na} D := (W^* \ot 1) \T{diag}(D) (W \ot 1) : \sD( 1 \ot_{\Na} D ) \to X \hot_A Y
\]
is symmetric.
\end{prop}
\begin{proof}
This follows since $\T{diag}(D) : \sD(\T{diag}(D)) \to Y^\infty$ is selfadjoint. Indeed,
\[
\begin{split}
\binn{(1 \ot_{\Na} D) \si, \te} & = \binn{\T{diag}(D) (W \ot 1) \si, (W \ot 1) \te } = \binn{\si,(W^* \ot 1) \T{diag}(D) (W \ot 1) \te} \\
& = \binn{\si, (1 \ot_\Na D) \te}
\end{split}
\]
for all $\si,\te \in \sD(1 \ot_\Na D)$.
\end{proof}


We remark that the symmetric lift only depends on $D : \sD(D) \to Y$ and the bounded adjointable isometry $W : X \to H_A$. It does not depend on the right $A_\de$-submodule $\sX \su X$ defined in \eqref{eq:defc1m}. The existence of $\sX$ is however crucial for proving that the symmetric lift is densely defined.

The final result of this section relates the symmetric lifts to the Gra\ss mann $\de$-connection. Thus, let $\Na_\de : \sX \to X \hot_A \Om_\de(\sA)$ denote the Gra\ss mann connection, see Definition \ref{d:gracon}.

\begin{lemma}\label{l:symgra}
Let $\si =  \xi \ot \eta \in \sX \ot_{A_\de} \sD(D)$. Then $\si \in \sD(1 \ot_{\Na} D)$ and 
\[
(1 \ot_\Na D)(\si) = \Na_\de(\xi)(\eta) + \xi \ot D \eta
\]
\end{lemma}
Remark that we have tacitly identitifed $\si$ with its image in $X \hot_A Y$.

\begin{proof}
By the proof of Lemma \ref{l:densymlif} we have that $\si \in \sD(1 \ot_{\Na} D)$ and that
\[
\begin{split}
(1 \ot_{\Na} D)(\si) & = (W^* \ot 1) \T{diag}(D) (W \ot 1)(\si) \\
& = (W^* \ot 1) \Big( \big\{\de(\inn{\ze_k,\xi}) \eta \big\}_{k = 1}^\infty + (W \ot 1)(\xi \ot D\eta) \Big) \\
& = \sum_{k = 1}^\infty \ze_k \ot \de(\inn{\ze_k,\xi})(\eta) + \xi \ot D \eta
\end{split}
\]
But this proves the lemma since $\sum_{k = 1}^\infty \ze_k \ot \de(\inn{\ze_k,\xi})(\eta) = \Na_\de(\xi)(\eta)$.
\end{proof}

In order to give the reader some feeling for what might be expected from symmetric lifts, we end this section by giving a basic example.

\subsection{Example: The half-line}
Let us consider the case where $X = C_0\big( (0,\infty)\big)$ consists of continuous functions on the half-line which vanish at $0$ and at $\infty$. We may then give $X$ the structure of a Hilbert $C^*$-module over the $C^*$-algebra $A = C_0(\rr)$ of continuous functions on the real line which vanish at $\pm \infty$. On top of this, we let $L^2(\rr)$ be the Hilbert space of (equivalence classes of) square integrable functions on the real line. This Hilbert space comes equipped with an injective $*$-homomorphism $\rho : C_0(\rr) \to \sL(L^2(\rr))$ given by point-wise multiplication $\rho(f)(\xi) := f \cd \xi$. Furthermore, we let $D : \sD(D) \to L^2(\rr)$ denote the unbounded selfadjoint operator obtained as the closure of the Dirac operator
\[
i \frac{d}{dt} : C_c^\infty(\rr) \to L^2(\rr)
\]
where $C_c^\infty(\rr) \su L^2(\rr)$ denotes the smooth compactly supported functions defined on $\rr$. We define the dense $*$-subalgebra $A_\de \su A$, by
\[
A_\de := \big\{ f \in C_0(\rr) \mid f \T{ is differentiable with } \frac{df}{dt} \in C_0(\rr) \big\} 
\]

The Hilbert $C^*$-module $X = C_0\big( (0,\infty)\big)$ is then generated by a single element. Indeed, we may choose a nowhere-vanishing differentiable function $\xi : (0,\infty) \to [0,1]$ such that $\xi, \frac{d \xi}{dt} \in X$. We then have that
\[
X = \T{cl}\big\{ \xi \cd f \mid f \in A \big\} \q \T{and} \q \inn{\xi,\xi} = \xi^2 \in A_\de
\]
where $\T{cl}(\cd)$ refers to the closure in supremum-norm. We may finally arrange that
\[
\| \inn{\xi,\xi} \|_\de = \sup_{t \in \rr} | \xi^2(t) | + 2 \sup_{t \in \rr} |(\xi \cd \frac{d \xi}{dt})(t) | \leq 1
\]

The bounded adjointable isometry $W : X \to H_A$ is then given by
\[
W : g \mapsto \big\{ \sqrt{H_n} \cd \inn{\xi,g} \big\}_{n = 1}^\infty = \big\{ (1 + n \xi^2)^{-1/2} (1 + (n-1) \xi^2)^{-1/2} \xi \cd g \big\}_{n = 1}^\infty
\]
and the bounded adjointable positive operator $K : H_A \to H_A$ is given by
\[
K : \{f_n\}_{n = 1}^\infty \mapsto \{ \xi^2 \cd f_n \}_{n = 1}^\infty
\]
The dense $A_\de$-submodule $\sX \su X$ is defined as the image $\sX := \T{Im}\big( W^* K^2 : H_{A_\de} \to X \big)$. It is then not hard to see that we have the inclusion 
\[
C_c^\infty\big( (0,\infty) \big) \su \sX
\]

The interior tensor product $X \hot_A L^2(\rr)$ is unitarily isomorphic to the Hilbert space $L^2\big( (0,\infty) \big)$ of square integrable functions on the half-line. Under this isomorphism the isometry $W \ot 1 : L^2\big( (0,\infty) \big) \to H_{L^2(\rr)}$ is given by
\[
W \ot 1 : g \mapsto \big\{ (1 + n \xi^2)^{-1/2} (1 + (n-1) \xi^2)^{-1/2} \xi \cd g \big\}_{n = 1}^\infty
\]

We are interested in obtaining a better understanding of the symmetric lift
\[
1 \ot_{\Na} D := (W^* \ot 1) \T{diag}(D) (W \ot 1) : \sD( 1 \ot_\Na D ) \to L^2\big( (0,\infty) \big)
\]
We first note that it follows by the proof of Lemma \ref{l:symgra} and the inclusion $C_c^\infty\big( (0,\infty) \big) \su \sX$ that
\[
C_c^\infty\big( (0,\infty) \big) \su \sD( 1 \ot_\Na D )
\]

Now, for each $g \in C_c^\infty\big( (0,\infty) \big)$ we may compute as follows:
\[
\begin{split}
(1 \ot_{\Na} D)(g) & = i \sum_{n = 1}^\infty \xi \sqrt{H_n} \frac{d}{dt}\big( \xi \sqrt{H_n} g \big)
= i \sum_{n = 1}^\infty \big( \xi^2 \cd H_n \cd \frac{d g}{dt} + 1/2 \cd g \cd \frac{d(\xi^2 \cd H_n)}{dt} \big) \\
& = i \frac{dg}{dt} + i/2 \cd \lim_{N \to \infty} \big( g \cd \frac{d(\xi^2 \cd (\xi^2 + 1/N)^{-1})}{dt} \big) \\
& = i \frac{dg}{dt} - i/2 \cd \lim_{N \to \infty} \big( g/N \cd \frac{d( (\xi^2 + 1/N)^{-1})}{dt} \big)
= i \frac{dg}{dt} 
\end{split}
\]
where the limit is taken in the norm on $L^2\big( (0,\infty) \big)$.

Thus, we obtain that $1 \ot_{\Na} D$ is a symmetric extension of the Dirac operator
\[
\dir := i \frac{d}{dt} : C_c^\infty\big(  (0,\infty)\big) \to L^2\big( (0,\infty)\big)
\]

Now, it is easily verified that $\T{Ker}\big( i + \dir^* \big) = \cc \cd \exp(-t)$ and that $\T{Ker}\big( i - \dir^* \big) = \{0\}$. It thus follows by \cite[Chapter X.1, Corollary]{SiRe:FS} that $1 \ot_{\Na} D$ is \emph{not} essentially selfadjoint, since $\dir : C_c^\infty\big( (0,\infty) \big) \to L^2\big( (0,\infty) \big)$ has no selfadjoint extensions.

\section{Compositions of regular unbounded operators}\label{s:comregunb}
\emph{Throughout this section, $X$ will be a Hilbert $C^*$-module over a $C^*$-algebra $A$, $D : \sD(D) \to X$ will be a selfadjoint, regular operator on $X$, and $x \in \sL(X)$ will be a bounded selfadjoint unbounded operator on $X$ such that:
\[
x \xi \in \sD(D) \E{ for all } \xi \in \sD(D) \, \, \, \, \, \E{and} \, \, \, \, \, [D,x] : \sD(D) \to X \E{ is bounded}
\]
The bounded extension of $[D,x]$ will be denoted by $\de(x)$.}

We remark that $\de(x)$ is automatically adjointable with $\de(x)^* = -\de(x)$. 
\medskip

\emph{The aim of this section is to study the regularity of the compositions $Dx$, $\T{cl}(x D)$, and $\T{cl}(x D x)$, where $\T{cl}(\dir)$ refers to the closure of an unbounded closable operator $\dir : \sD(\dir) \to X$. This regularity issue has been studied in detail by S. L. Woronowicz under the assumption that $x$ is invertible, see \cite[Section 2, Example 2 and 3]{Wor:UAQ}.}
\medskip

The general investigations of this section will allow us to obtain a better understanding of the symmetric lift introduced in Section \ref{s:symsel}.

Our main tool is the local-global principle for regular operators, see \cite[Theorem 4.2]{KaLe:LGR}. For the readers convenience we now recall the statement of this result: Let $\dir : \sD(\dir) \to X$ be a closed unbounded operator wih a densely defined adjoint $\dir^*$. For each state $\rho : A \to \cc$ we have the \emph{localization} $X_\rho$ of $X$. This is the Hilbert space obtained as the completion of $X/N_\rho$ with respect to the inner product $\inn{[\xi],[\eta]}_\rho := \rho(\inn{\xi,\eta})$, where $N_\rho := \{ \xi \in X \, | \, \rho(\inn{\xi,\xi}) = 0 \}$. The unbounded operator $\dir$ then induces an unbounded operator on $X_\rho$, 
\[
\dir_\rho : \sD(\dir_\rho) \to X_\rho \q [\xi] \mapsto [\dir \xi]
\]
with domain $\sD(\dir_\rho)$ defined as the image of $\sD(\dir)$ in $X_\rho$. The \emph{localization} of $\dir$ at the state $\rho$ is the unbounded operator $\T{cl}(\dir_{\rho})$.

\begin{thm}[Local-global principle]\label{t:locglopri}
The closed unbounded operator $\dir : \sD(\dir) \to X$ with densely defined adjoint $\dir^*$ is regular if and only if
\[
(\dir_{\rho})^* = \T{cl}\big( (\dir^*)_{\rho} \big)
\]
for all states $\rho : A \to \cc$.
\end{thm}

We now study the regularity of the unbounded operator $Dx : \sD(Dx) \to X$ with domain $\sD(Dx) := \{ \xi \in X \, | \, x \xi \in \sD(D) \}$. We remark that $Dx$ is already closed. The next to lemmas serve to compute the adjoint of $Dx$.
 
\begin{lemma}\label{l:incadj}
\[
Dx - \de(x) \su (Dx)^*
\]
\end{lemma}
\begin{proof}
Let $\xi, \eta \in \sD(Dx)$. We then have that
\[
\begin{split}
\inn{Dx \xi,\eta} 
& = \lim_{n \to \infty} \inn{Dx \xi, i(i + D/n)^{-1} \eta} 
= \lim_{n \to \infty} \inn{\xi, i x D(i+ D/n)^{-1} \eta} \\
& = - \inn{\xi, \de(x) \eta} + \lim_{n \to \infty} \inn{\xi,i Dx(i + D/n)^{-1}\eta} \\
& = - \inn{\xi, \de(x) \eta} + \inn{\xi,Dx \eta}
+ \lim_{n \to \infty} \inn{\xi,i D/n \cd (i + D/n)^{-1} \de(x) (i + D/n)^{-1} \eta}
\end{split}
\]
It therefore suffices to show that
\[
i D/n \cd (i + D/n)^{-1} \de(x) (i + D/n)^{-1} \eta \to 0
\]
But this follows easily since
\[
\begin{split}
& i D/n \cd (i + D/n)^{-1} \de(x) (i + D/n)^{-1} \eta \\
& \q = \de(x) i (i + D/n)^{-1}\eta + (i + D/n)^{-1} \de(x) (i + D/n)^{-1}\eta
\end{split}
\]
\end{proof}

In order to prove the other inclusion $(Dx)^* \su Dx - \de(x)$, we remark that the adjoint of $x D : \sD(D) \to X$ is precisely the unbounded operator $Dx$. This follows from the selfadjointness of $D : \sD(D) \to X$ and $x \in \sL(X)$.

\begin{lemma}\label{l:adjinc}
\[
(Dx)^* \su Dx - \de(x)
\]
\end{lemma}
\begin{proof}
Notice that $x D + \de(x) \su Dx$. But this implies that $(Dx)^* \su (x D + \de(x))^* = Dx - \de(x)$.
\end{proof}

We want to apply the local global principle for regular operators to show that $Dx : \sD(Dx) \to X$ is regular. Thus, we need to compute the localization $\T{cl}\big( (Dx)_\rho \big)$ and its adjoint $\big( (Dx)_\rho \big)^*$ for an arbitrary state $\rho : A \to \cc$. This is the content of the next lemma.

To ease the notation, let $y \ot 1 \in \sL(X_\rho)$ denote the closure of $y_\rho$ for a bounded adjointable operator $y : X \to X$.

\begin{lemma}\label{l:loccloadj}
Let $\rho : A \to \cc$ be a state. Then we have the identities
\[
\T{cl}\big( (Dx)_\rho \big) = \T{cl}( D_\rho) (x \ot 1) 
\q \T{and} \q \big( (Dx)_\rho \big)^* = \T{cl}(D_\rho)(x \ot 1) - \de(x) \ot 1 
\]
\end{lemma}
\begin{proof}
Remark first that $(Dx)_\rho \su \T{cl}(D_\rho)(x \ot 1)$. This implies the inclusion
\[
\T{cl}\big( (Dx)_\rho \big) \su \T{cl}(D_\rho)(x \ot 1)
\]
Furthermore, since $\big( \T{cl}(D_\rho)(x \ot 1) \big)^* = \T{cl}(D_\rho)(x \ot 1) - \de(x) \ot 1$ by Lemma \ref{l:incadj} and Lemma \ref{l:adjinc}, we get that
\[
\T{cl}(D_\rho)(x \ot 1) - \de(x) \ot 1 \su \big( (Dx)_\rho \big)^*  
\]

To prove the reverse inclusions, note that $x D + \de(x) \su Dx$. This implies that $(x \ot 1) D_{\rho} + \de(x) \ot 1 \su (D x)_\rho$. We may then deduce that
\[
\big( (Dx)_\rho \big)^* \su \big( (x \ot 1) D_\rho + \de(x) \ot 1 \big)^* = \T{cl}(D_\rho)(x \ot 1) - \de(x) \ot 1
\]

We have thus proved the identity
\[
\big( (Dx)_\rho \big)^*  = \T{cl}(D_\rho)(x \ot 1) - \de(x) \ot 1
\]
But it then follows, since $X_\rho$ is a Hilbert space, that
\[
\T{cl}\big( (Dx)_\rho \big) = \big( (Dx)_\rho \big)^{**} = \T{cl}(D_\rho)(x \ot 1)
\]

This proves the lemma.
\end{proof}

We are now ready to prove the first main result of this section:

\begin{prop}\label{p:rigcomreg}
The closed unbounded operator $Dx : \sD(Dx) \to X$ is regular and the adjoint is given by $(Dx)^* = Dx - \de(x) : \sD(Dx) \to X$.
\end{prop}
\begin{proof}
The formula for the adjoint $(Dx)^*$ is a consequence of Lemma \ref{l:incadj} and Lemma \ref{l:adjinc}.

Let now $\rho : A \to \cc$ be a state. By Theorem \ref{t:locglopri} we need only show that
\begin{equation}\label{eq:locrigcom}
\big( (Dx)_\rho \big)^* = \T{cl}\big( \big( (Dx)^*\big)_\rho \big)
\end{equation}

Applying Lemma \ref{l:loccloadj} we obtain that
\[
\big( (Dx)_\rho \big)^* = \T{cl}(D_\rho)(x \ot 1) - \de(x) \ot 1
\]

By another application of Lemma \ref{l:loccloadj} we get that
\[
\T{cl}\big( \big( (Dx)^* \big)_\rho \big) = \T{cl}\big( (Dx)_\rho - \de(x)_\rho \big) = \T{cl}(D_\rho)(x \ot 1) - \de(x) \ot 1
\]

This proves the identity in \eqref{eq:locrigcom} and thereby also the result of the proposition.
\end{proof}

We may now treat the regularity problem for the composition $x D : \sD(D) \to X$. This is carried out in the next proposition. We recall that $(x D)^* = Dx : \sD(Dx) \to X$. This does however not imply the regularity of $\T{cl}(x D)$. Indeed, it is possible to construct a closed unbounded, \emph{non-regular} operator $\dir : \sD(\dir) \to X$ with a \emph{regular} adjoint $\dir^* : \sD(\dir^*) \to X$, see \cite[Proposition 2.3]{Pal:ROH} and \cite[Proposition 6.3]{KaLe:LGR}. Thus, the result in \cite[Corollary 9.6]{Lan:HCM} is incorrect. None-the-less we have the following:

\begin{prop}\label{p:lefcomreg}
The closure $\T{cl}(x D)$ is regular and given by $\T{cl}(xD) = D x - \de(x) : \sD(Dx) \to X$. 
\end{prop}
\begin{proof}
Let $\rho : A \to \cc$ be a state. By the local-global principle in Theorem \ref{t:locglopri}, the regularity of $\T{cl}(xD)$ will follow from the identity
\begin{equation}\label{eq:leflocadj}
\big( \big( \T{cl}(xD) \big)_\rho \big)^* = \T{cl}\big( \big( (xD)^* \big)_\rho \big)
\end{equation}

The left hand side of \eqref{eq:leflocadj} can be rewritten as
\[
\big( \big( \T{cl}(xD) \big)_\rho \big)^* = \big( (x \ot 1) \T{cl}(D_\rho) \big)^* = \T{cl}(D_\rho) (x \ot 1)
\]
where the first identity follows since $\big( \T{cl}(xD) \big)_{\rho}$ and $(x \ot 1) \T{cl}(D_\rho)$ agrees on the subspace $\sD(D_\rho) \su X_\rho$ and the second identity follows from the regularity and selfadjointness of $D : \sD(D) \to X$.

The right hand side of \eqref{eq:leflocadj} can be computed using Lemma \ref{l:loccloadj}. We obtain that
\[
\T{cl}\big( \big( (xD)^* \big)_\rho \big) = \T{cl}\big( (Dx)_\rho \big) =  \T{cl}(D_\rho) (x \ot 1)
\]
This proves the identity in \eqref{eq:leflocadj} and thus that $\T{cl}(x D)$ is regular.

Now, since $\T{cl}(x D)$ is regular we have that $\T{cl}(x D) = (x D)^{**} = (D x)^* = Dx - \de(x)$, see \cite[Corollary 9.4]{Lan:HCM}. This proves the last part of the proposition.
\end{proof}

We conclude this section by showing that $x D x : \sD(Dx) \to X$ is essentially selfadjoint and regular, thus the closure $\T{cl}(x D x)$ is selfadjoint and regular.

\begin{prop}\label{p:essregcom}
The closure $\T{cl}(x D x)$ is selfadjoint and regular and given by $\T{cl}(x D x) = D x^2 - \de(x) x : \sD(Dx^2) \to X$.
\end{prop}
\begin{proof}
By Proposition \ref{p:rigcomreg}, $Dx : \sD(Dx) \to X$ is regular with $(Dx)^* = Dx - \de(x) : \sD(Dx) \to X$. This fact is equivalent to the selfadjointness and regularity of the anti-diagonal unbounded operator
\[
\ma{cc}{0 & Dx - \de(x) \\ Dx & 0} : \sD(Dx) \op \sD(Dx) \to X \op X
\]
see \cite[Lemma 2.3]{KaLe:LGR}. It therefore follows by Proposition \ref{p:lefcomreg} that
\[
\ma{cc}{0 & \T{cl}(x D x) - x\de(x) \\ 
\T{cl}(x Dx) & 0} : \sD(\T{cl}(x D x)) \op \sD(\T{cl}(x D x)) \to X \op X
\]
is regular. Furthermore, we have that
\[
\ma{cc}{0 & \T{cl}(x D x) - x\de(x) \\ 
\T{cl}(x Dx) & 0} = 
\ma{cc}{0 & Dx^2 - \de(x) x \\ D x^2 & 0} - \ma{cc}{0 & x \de(x) \\ \de(x)x & 0}
\]

We may thus conclude that $\T{cl}(x D x) = D x^2 - \de(x) x : \sD(D x^2) \to X$. It then follows by Proposition \ref{p:rigcomreg} that $\T{cl}(x D x)$ is regular. Furthermore, the adjoint is given by
\[
(x D x)^* = (Dx^2)^* + x \de(x) = D x^2 - \de(x^2) + x \de(x) = D x^2 - \de(x) x
\]
This shows that $\T{cl}(x D x)$ is also selfadjoint and the proposition is proved.
\end{proof}

\section{Selfadjointness and regularity of lifts}\label{s:selreg}
\emph{We will now return to the setting described in the beginning of Section \ref{s:symsel}. Furthermore, we let $W : X \to H_A$ and $K : H_A \to H_A$ be as in Theorem \ref{t:difabs}, and as in Remark \ref{r:diffra} we let $\{\ze_k\}_{k = 1}^\infty$ be a sequence in $X$ such that $W(\eta) = \{\inn{\ze_k,\eta} \}_{k = 1}^\infty$ for all $\eta \in X$.}

We recall that $W^* K W : X \to X$ has dense image and it thus follows that
\[
\De := (W^* K W)^2 \ot 1 = (W^* K^2 W) \ot 1 : X \hot_A Y \to X \hot_A Y
\]
has dense image as well.

\emph{We are interested in proving that the composition
\[
\De (1 \ot_{\Na} D) \De : \sD\big( \T{diag}(D)(W \ot 1) \De \big) \to X \hot_A Y
\]
is an essentially selfadjoint and regular unbounded operator.}

We first notice that the map $\io : M_\infty(\sL(Y)) \to \sL(Y^\infty)$ given by
\[
\io( \{ T_{ij} \})(\{\eta_n\}) := \{ \sum_{j = 1}^\infty T_{ij}(\eta_j) \}_{i = 1}^\infty \q \{T_{ij}\} \in M_\infty(\sL(Y)) \, \, , \, \, \{\eta_n\} \in Y^\infty
\]
induces an injective $*$-homomorphism $\io : \sK\big( H_{\sL(Y)}\big) \to \sL(Y^\infty)$. In particular, we have that $\| \io(T) \| = \| T \|$ for all $T \in \sK(H_{\sL(Y)})$. This enables us to prove the following:

\begin{lemma}\label{l:comcombou}
Let $T \in \sK(H_A)_\de$. Then $\io( \rho(T)) \in \sL(Y^\infty)$ preserves the domain of $\T{diag}(D)$ and $\io( \de(T) ) \in \sL(Y^\infty)$ is an extension of the commutator
\[
\big[ \T{diag}(D), \io(\rho(T)) \big] : \sD(\T{diag}(D)) \to Y^\infty 
\]
\end{lemma}
\begin{proof}
Let $\eta = \{ \eta_n \} \in \sD(\T{diag}(D))$.

Suppose first that $T \in M_\infty(\sA)$. Then clearly $\io(\rho(T))(\eta) = \big\{ \sum_{j = 1}^\infty \rho(x_{ij}) \eta_j \} \in \sD(\T{diag}(D))$ and furthermore
\[
\big[ \T{diag}(D), \io (\rho(T)) \big](\eta) = \big\{ \sum_{j = 1}^\infty [D,\rho(x_{ij})](\eta_j) \big\} = \io(\delta(T))(\eta)
\]
This proves the claim of the lemma in this case.

For a general $T \in \sK(H_A)_\de$, we may choose a sequence $\{ T_m\}$ in $M_\infty(\sA)$ such that $T_m \to T$ in the norm $\| \cd \|_\de : \sK(H_A)_\de \to [0, \infty)$. We then use the fact that $\T{diag}(D) : \sD(\T{diag}(D)) \to Y^\infty$ is closed to conclude that $\io(\rho(T))(\eta) \in \sD(\T{diag}(D))$ with
\[
\sD(\T{diag}(D))\big( \io(\rho(T))(\eta) \big) = \io(\rho(T))\big( \T{diag}(D)(\eta) \big) + \io(\de(T))(\eta)
\]
This proves the lemma.
\end{proof}

Let us consider the bounded positive selfadjoint operator
\[
\De_W := (W \ot 1) \De (W^* \ot 1) : (P \ot 1) Y^\infty \to (P \ot 1) Y^\infty
\]
where $P \ot 1 := (W \ot 1)(W^* \ot 1) : Y^\infty \to Y^\infty$ is the orthogonal projection associated with the isometry $(W \ot 1) : X \hot_A Y \to Y^\infty$, see Section \ref{s:symsel}.

We then remark that $\De (1 \ot_\Na D) \De : \sD\big( \T{diag}(D)(W \ot 1)\De \big) \to X \hot_A Y$ and 
\[
\De_W \T{diag}(D) \De_W : \sD\big( \T{diag}(D) \De_W \big) \to (P \ot 1) Y^\infty
\]
are unitarily equivalent unbounded operators. Furthermore, we have that
\[
\begin{split}
\De_W & = (W \ot 1)(W^* K^2 W \ot 1)(W^* \ot 1) \\
& = \io(\rho(P K^2))\big|_{(P \ot 1) Y^\infty} : (P \ot 1) Y^\infty \to (P \ot 1) Y^\infty
\end{split}
\]

\begin{prop}
The unbounded operator $\De_W \T{diag}(D) \De_W : \sD\big( \T{diag}(D) \De_W \big) \to (P \ot 1) Y^\infty$ is essentially selfadjoint and regular.
\end{prop}
\begin{proof}
It is enough to show that
\[
\io( PK^2) \T{diag}(D) \io(P K^2) : \sD\big( \T{diag}(D) \De_W \big) + \big( (1 - P) \ot 1 \big) Y^\infty \to Y^\infty
\]
is essentially selfadjoint and regular. Now, by the differentiable absorption theorem (Theorem \ref{t:difabs}), we have that $P K^2 \in \sK(H_A)_\de$. By Lemma \ref{l:comcombou}, the pair consisting of the unbounded selfadjoint regular operator $\T{diag}(D) : \sD(\T{diag}(D)) \to Y^\infty$ and the bounded selfadjoint operator $\io( \rho(P K^2) ) : Y^\infty \to Y^\infty$ therefore satisfies the assumptions applied in Section \ref{s:comregunb}. This proves the current lemma by an application of Proposition \ref{p:essregcom}.
\end{proof}

The main result of this section now follows immediately:

\begin{thm}
The unbounded operator $\De (1 \ot_\Na D) \De : \sD\big( (1 \ot_\Na D) \De \big) \to X \hot_A Y$ is essentially selfadjoint and regular.
\end{thm}

\providecommand{\bysame}{\leavevmode\hbox to3em{\hrulefill}\thinspace}
\providecommand{\MR}{\relax\ifhmode\unskip\space\fi MR }
\providecommand{\MRhref}[2]{%
  \href{http://www.ams.org/mathscinet-getitem?mr=#1}{#2}
}
\providecommand{\href}[2]{#2}

\bibliographystyle{amsalpha-lmp}

\begin{thebibliography}{\textsc{CuQu95}}

\bibitem[\textsc{Bla98}]{Bla:KOA}
\textsc{B.~Blackadar}, \emph{{$K$}-theory for operator algebras}, second ed.,
  Mathematical Sciences Research Institute Publications, vol.~5, Cambridge
  University Press, Cambridge, 1998. \MR{1656031 (99g:46104)}

\bibitem[\textsc{BlCu91}]{BlCu:DNS}
\textsc{B.~Blackadar} and \textsc{J.~Cuntz}, \emph{Differential {B}anach
  algebra norms and smooth subalgebras of {$C\sp *$}-algebras}, J. Operator
  Theory \textbf{26} (1991), no.~2, 255--282. \MR{1225517 (94f:46094)}

\bibitem[\textsc{Ble96}]{Ble:GHM}
\textsc{D.~P. Blecher}, \emph{A generalization of {H}ilbert modules}, J. Funct.
  Anal. \textbf{136} (1996), no.~2, 365--421. \MR{1380659 (97g:46071)}

\bibitem[\textsc{Ble97}]{Ble:AHM}
\bysame, \emph{A new approach to {H}ilbert {$C\sp *$}-modules}, Math. Ann.
  \textbf{307} (1997), no.~2, 253--290. \MR{1428873 (98d:46063)}

\bibitem[\textsc{BMS13}]{BrMeSu:GSU}
\textsc{S.~Brain}, \textsc{B.~Mesland}, and \textsc{W.~D.~V. Suijlekom},
  \emph{Gauge theory for spectral triples and the unbounded {K}asparov
  product},  \texttt{arXiv:1306.1951 [math.KT]}.

\bibitem[\textsc{Con85}]{Con:NDG}
\textsc{A.~Connes}, \emph{Noncommutative differential geometry}, Inst. Hautes
  \'Etudes Sci. Publ. Math. (1985), no.~62, 257--360. \MR{823176 (87i:58162)}

\bibitem[\textsc{Con94}]{Con:NCG}
\bysame, \emph{Noncommutative geometry}, Academic Press, Inc., San Diego, CA,
  1994. \MR{1303779 (95j:46063)}

\bibitem[\textsc{CuQu95}]{CuQu:AEN}
\textsc{J.~Cuntz} and \textsc{D.~Quillen}, \emph{Algebra extensions and
  nonsingularity}, J. Amer. Math. Soc. \textbf{8} (1995), no.~2, 251--289.
  \MR{1303029 (96c:19002)}

\bibitem[\textsc{FrLa02}]{FrLa:FHM}
\textsc{M.~Frank} and \textsc{D.~R. Larson}, \emph{Frames in {H}ilbert {$C\sp
  \ast$}-modules and {$C\sp \ast$}-algebras}, J. Operator Theory \textbf{48}
  (2002), no.~2, 273--314. \MR{1938798 (2003i:42040)}

\bibitem[\textsc{JeTh91}]{JeTh:EKT}
\textsc{K.~K. Jensen} and \textsc{K.~Thomsen}, \emph{Elements of
  {$KK$}-theory}, Mathematics: Theory \& Applications, Birkh\"auser Boston,
  Inc., Boston, MA, 1991. \MR{1124848 (94b:19008)}

\bibitem[\textsc{Kaa13}]{Kaa:SSB}
\textsc{J.~Kaad}, \emph{A {S}erre-{S}wan theorem for bundles of bounded
  geometry}, J. Funct. Anal. \textbf{265} (2013), no.~10, 2465--2499.
  \MR{3091822}

\bibitem[\textsc{KaLe12}]{KaLe:LGR}
\textsc{J.~Kaad} and \textsc{M.~Lesch}, \emph{A local global principle for
  regular operators in {H}ilbert {$C\sp *$}-modules}, J. Funct. Anal.
  \textbf{262} (2012), no.~10, 4540--4569. \MR{2900477}

\bibitem[\textsc{KaLe13}]{KaLe:SFU}
\bysame, \emph{Spectral flow and the unbounded {K}asparov product}, Adv. Math.
  \textbf{248} (2013), 495--530. \MR{3107519}

\bibitem[\textsc{Kar87}]{Kar:HCK}
\textsc{M.~Karoubi}, \emph{Homologie cyclique et {$K$}-th\'eorie}, Ast\'erisque
  (1987), no.~149, 147. \MR{913964 (89c:18019)}

\bibitem[\textsc{Kas80a}]{Kas:HSV}
\textsc{G.~G. Kasparov}, \emph{Hilbert {$C\sp{\ast} $}-modules: theorems of
  {S}tinespring and {V}oiculescu}, J. Operator Theory \textbf{4} (1980), no.~1,
  133--150. \MR{587371 (82b:46074)}

\bibitem[\textsc{Kas80b}]{Kas:OFE}
\bysame, \emph{The operator {$K$}-functor and extensions of {$C\sp{\ast}
  $}-algebras}, Izv. Akad. Nauk SSSR Ser. Mat. \textbf{44} (1980), no.~3,
  571--636, 719. \MR{582160 (81m:58075)}

\bibitem[\textsc{Lan95}]{Lan:HCM}
\textsc{E.~C. Lance}, \emph{Hilbert {$C\sp *$}-modules}, London Mathematical
  Society Lecture Note Series, vol. 210, Cambridge University Press, Cambridge,
  1995, A toolkit for operator algebraists. \MR{1325694 (96k:46100)}

\bibitem[\textsc{Mes14}]{Mes:UCN}
\textsc{B.~Mesland}, \emph{Unbounded bivariant {$K$}-theory and correspondences
  in noncommutative geometry}, J. Reine Angew. Math. \textbf{691} (2014),
  101--172. \MR{3213549}

\bibitem[\textsc{MiPh84}]{MiPh:ETH}
\textsc{J.~A. Mingo} and \textsc{W.~J. Phillips}, \emph{Equivariant triviality
  theorems for {H}ilbert {$C\sp{\ast} $}-modules}, Proc. Amer. Math. Soc.
  \textbf{91} (1984), no.~2, 225--230. \MR{740176 (85f:46111)}

\bibitem[\textsc{Pal99}]{Pal:ROH}
\textsc{A.~Pal}, \emph{Regular operators on {H}ilbert {$C\sp *$}-modules}, J.
  Operator Theory \textbf{42} (1999), no.~2, 331--350. \MR{1716957
  (2000h:46072)}

\bibitem[\textsc{RaTh03}]{RaTh:KSF}
\textsc{I.~Raeburn} and \textsc{S.~J. Thompson}, \emph{Countably generated
  {H}ilbert modules, the {K}asparov stabilisation theorem, and frames with
  {H}ilbert modules}, Proc. Amer. Math. Soc. \textbf{131} (2003), no.~5,
  1557--1564 (electronic). \MR{1949886 (2003j:46089)}

\bibitem[\textsc{ReSi75}]{SiRe:FS}
\textsc{M.~Reed} and \textsc{B.~Simon}, \emph{Methods of modern mathematical
  physics. {II}. {F}ourier analysis, self-adjointness}, Academic Press
  [Harcourt Brace Jovanovich, Publishers], New York-London, 1975. \MR{0493420
  (58 \#12429b)}

\bibitem[\textsc{Wor91}]{Wor:UAQ}
\textsc{S.~L. Woronowicz}, \emph{Unbounded elements affiliated with {$C\sp
  *$}-algebras and noncompact quantum groups}, Comm. Math. Phys. \textbf{136}
  (1991), no.~2, 399--432. \MR{1096123 (92b:46117)}

\end{thebibliography}

\end{document}